\documentclass{amsart}
\usepackage{url}
\usepackage{amssymb,amscd,amsmath,amsthm,amsfonts}
\usepackage[all]{xy}

\newcommand\scl{\operatornamewithlimits{scl}}
\newcommand\comp{\operatornamewithlimits{c}}

\newcommand\boxb[1]{\square_b}

\numberwithin{equation}{section}

\newcommand\paperbody%
        {}
\newcommand\TT{\mathbb{T}}
\newcommand{\smA}{\mathcal{S}}
\newcommand{\cK}{{\mathcal K}}
\newcommand{\Index}{{\rm Index}}
\newcommand{\Kc}{\operatorname{K}_{\text{c}}}

\newcommand{\Kt}{\operatorname{K}}
\newcommand{\Ko}{\operatorname{K}^{1}}
\newcommand{\Kco}{\operatorname{K}^{1}_{\text{c}}}

\newcommand{\Ho}{\operatorname{H^{\odd}}}

\newcommand{\Hoc}{\operatorname{H^{\odd}_{\text{c}}}}

\newcommand{\cU}{{\mathcal U}}
\newcommand{\cN}{{\mathcal N}}

\newcommand{\R}{{\mathbb R}}

\newtheorem{lemma}{Lemma}
\newtheorem{proposition}{Proposition}
\newtheorem{corollary}{Corollary}
\newtheorem{theorem}{Theorem}
\newtheorem{non-theorem}{Non-Theorem}

\theoremstyle{remark}
\newtheorem{definition}{Definition}

\newcommand\coF{{}^{\mathcal{C}}\kern-2pt\Lambda}

\newcommand\cFTs{{}^{\Phi}\overline{T}\kern-1pt{}^*}

\newcommand\even{\text{even}}
\newcommand\odd{\text{odd}}

\newcommand\Cl{\operatorname{Cl}}

\newcommand\Ch{\operatorname{Ch}}
\newcommand\Td{\operatorname{Td}}
\newcommand\com[1]{\overline{#1}}

\newcommand\cA{\mathcal{A}}

\newcommand\cH{\mathcal{H}}

\newcommand\cF{\mathcal{F}}

\newcommand\cL{\mathcal{L}}

\newcommand{\cP}{{\mathcal P}}

\newcommand\fM{\mathfrak{M}}

\newcommand\ZZ{\mathbb Z}

\newcommand\bbC{\mathbb C}

\newcommand\bbE{\mathbb E}

\newcommand\bbN{\mathbb N}
\newcommand\bbP{\mathbb P}
\newcommand\bbQ{\mathbb Q}
\newcommand\bbR{\mathbb R}
\newcommand\bbS{\mathbb S}
\newcommand\bbT{\mathbb T}
\newcommand\bbZ{\mathbb Z}

\newcommand\cM{\mathcal M}
\newcommand\cS{\mathcal S}
\newcommand\CIc{{\mathcal{C}}^{\infty}_c}

\newcommand\CI{{\mathcal{C}}^{\infty}}

\newcommand\Diag{\operatorname{Diag}}

\newcommand\Diff[1]{\operatorname{Diff}^{#1}}

\newcommand\cFNs{{}^{\Phi}\overline N\kern-1pt{}^*}

\newcommand\ind{\operatorname{ind}}

\newcommand\Thom{\operatorname{Thom}}

\newcommand\tr{\operatorname{tr}}

\newcommand\Hom{\operatorname{Hom}}
\newcommand\Id{\operatorname{Id}}

\newcommand\UU{\operatorname{U}}
\newcommand\PU{\operatorname{PU}}

\newcommand\ci{${\mathcal{C}}^\infty$}

\newcommand\dCI{\dot{\mathcal{C}}^{\infty}}

\newcommand\ha{\frac{1}{2}}

\newcommand\pa{\partial}

\newcommand\longhookrightarrow{\hookrightarrow}

\newcommand\Mand{\text{ and }}

\newcommand\Min{\text{ in }}

\newcommand\Mon{\text{ on }}

\newcommand\Mwhere{\text{ where }}
\newcommand\Mwith{\text{ with }}

\begin{document}
\title[Projective families with decomposable DD invariant]
{The index of projective families of elliptic operators: the decomposable case}


\author{V. Mathai}
\address{Department of Mathematics, University of Adelaide,
Adelaide 5005, Australia}
\email{vmathai@maths.adelaide.edu.au}
\author{R.B. Melrose}
\address{Department of Mathematics,
Massachusetts Institute of Technology,
Cambridge, Mass 02139}
\email{rbm@math.mit.edu}
\author{I.M. Singer}
\address{Department of Mathematics,
Massachusetts Institute of Technology,
Cambridge, Mass 02139}
\email{ims@math.mit.edu}


\begin{abstract} An index theory for projective families of elliptic
  pseudodifferential operators is developed under two conditions. First,
  that the twisting, i\@.e\@.~Dixmier-Douady, class is in
  $\operatorname{H}^2(X;\bbZ)\cup \operatorname{H}^1(X;\bbZ)\subset
  \operatorname{H}^3(X;\bbZ)$ and secondly that the 2-class part is
  trivialized on the total space of the fibration. One of the features of
  this special case is that the corresponding Azumaya bundle can be refined
  to a bundle of smoothing operators.  The topological and the analytic
  index of a projective family of elliptic operators associated with the
  smooth Azumaya bundle both take values in twisted $K$-theory of the
  parameterizing space and the main result is the equality of these two
  notions of index. The twisted Chern character of the index class is then
  computed by a variant of Chern-Weil theory.
\end{abstract}




\thanks{{\em Acknowledgments.} We would like to thank the referees for
  their helpful suggestions. The first author thanks the Australian Research Council
   for support.  The second author acknowledges the support of
  the National Science Foundation under grant DMS0408993.}

\maketitle

\section*{Introduction}

The basic object leading to twisted K-theory for a space, $X,$ can be taken
to be a principal $\PU$-bundle $\cP\longrightarrow X,$ where
$\PU=\UU(\cH)/\UU(1)$ is the group of projective unitary operators on some
separable infinite-dimensional Hilbert space $\cH.$ Circle bundles over $X$
are classified up to isomorphism by their Chern classes in
$\operatorname{H}^2(X;\bbZ)$ and analogously principal $\PU$ bundles are
classified by $\operatorname{H}^3(X;\bbZ)$ with the element $\delta(\cP)$
being the Dixmier-Douady invariant of $\cP.$ Just as $\Kt^0(X),$ the
ordinary K-theory group of $X,$ may be identified with the group of
homotopy classes of maps $X\longrightarrow \cF(\cH)$ into the Fredholm
operators on $\cH,$ the twisted K-theory group $\Kt^0(X;\cP)$ may be
identified with the homotopy classes of sections of the bundle
$\cP\times_{\PU}\cF$ arising from the conjugation action of $\PU$ on $\cF.$
The action of $\PU$ on the compact operators, $\cK,$ induces the
\emph{Azumaya} bundle, $\cA.$ The K-theory, in the sense of $C^*$ algebras,
of the space of continuous sections of this bundle, written $\Kt^0(X;\cA),$
is naturally identified with $\Kt^0(X;\cP).$ From an analytic viewpoint
$\cA$ is more convenient to deal with than $\cP$ itself.

In the case of circle bundles isomorphisms are classified up to homotopy by
an element of $\operatorname{H}^1(X;\bbZ),$ corresponding to the homotopy
class of a smooth map $X\longrightarrow \UU(1).$ Similarly, $\delta \in
\operatorname{H}^3(X;\bbZ)$ determines $\cP$ up to isomorphism with the
isomorphism class determined up to homotopy by an element of
$\operatorname{H}^2(X;\bbZ),$ corresponding to the fact that $\PU$ is a
$K(\bbZ,2).$ The result is that $\Kt^0(X;\cA)$ depends as a group on the
choice of Azumaya bundle with DD invariant $\delta$ up to an action of
$\operatorname{H}^2(X;\bbZ).$

In \cite{mms1} we extended the index theorem for a family of elliptic
operators, giving the equality of the analytic and the topological index
maps in K-theory, to the case of twisted K-theory where the twisting class
is a torsion element of $\operatorname{H}^3(X;\bbZ).$ In this paper we
prove a similar index equality in the case of twisted K-theory when the
index class is decomposable
\begin{equation}
\delta =\alpha \cup\beta,\ \alpha \in \operatorname{H}^1(X;\bbZ),\ \beta
\in \operatorname{H}^2(X;\bbZ),
\label{mms2002.137}\end{equation}
and the fibration $\phi:Y\longrightarrow X$ is such that $\phi^*\beta =0$
in $\operatorname{H}^2(Y;\bbZ).$

Under the assumption \eqref{mms2002.137}, that the class $\delta$ is
\emph{decomposed,} we show below that there is a choice of principal $\PU$
bundle with class $\delta$ such that the classifying map above,
$c_P:X\longrightarrow K(\bbZ;3)$ factors through $\UU(1)\times\PU.$
Twisting by a homotopically non-trivial map $\kappa:X\longrightarrow \PU$
does not preserve this property, so in this decomposed case there is indeed
a natural choice of smooth Azumaya bundle, $\smA,$ up to homotopically trivial
isomorphism and this induces a choice of twisted K-group determined by the
decomposition of $\delta;$ we denote this well-defined twisted K-group by
\begin{equation}
\Kt^0(X;\alpha,\beta )=\Kt^0(X;\cA),\quad \cA=\overline{\smA}.
\label{mms2002.197}\end{equation}

The effect on smoothness of the assumption of decomposability on the
Dixmier-Douady class can be appreciated by comparison with the simpler case
of degree 2. Thus, if $\alpha _1\cup\alpha _2\in \operatorname{H}^2(X;\bbZ)$ is a
decomposed class, $\alpha _i\in \operatorname{H}^1(X;\bbZ)$ for $i=1,2,$ then the
associated line bundle is the pull-back of the Poincar\'e line bundle
associated to a polarization on the 2-torus under the map $u_1\times u_2,$
where the $u_i\in\CI(X;\UU(1))$ represent the $\alpha _i.$ This is to be
contrasted with the general case in with the line bundle is the
pull-back from a classifying space such as $\PU,$ and is only unique up to 
twisting by a smooth map $\kappa':X\longrightarrow \UU(1).$

The data we use to define a smooth Azumaya bundle is:-
\begin{itemize}
\item 
A smooth function 
\begin{equation}
u\in\CI(X;\UU(1))
\label{mms2002.115}\end{equation}
the homotopy class of which represents $\alpha \in \operatorname{H}^1(X,\bbZ).$
\item A Hermitian line bundle (later with unitary connection)
\begin{equation}
\xymatrix{
L\ar[d]\\
X
}
\label{mms2002.86}\end{equation}
with Chern class $\beta \in \operatorname{H}^2(X;\bbZ).$
\item A smooth fiber bundle of compact manifolds
\begin{equation}
\xymatrix{
Z\ar@{-}[r]&Y\ar[d]^\phi\\
& X
}
\label{mms2002.59}\end{equation}
such that $\phi^*\beta =0$ in $\operatorname{H}^2(Y;\bbZ).$
\item An explicit global unitary trivialization 
\begin{equation}
\gamma:\phi^*(L)\overset{\simeq}\longrightarrow Y\times\bbC.
\label{mms2002.87}\end{equation}
\end{itemize}
These hypotheses are satisfied by taking $Y=\tilde L,$ the circle bundle
of $L,$ and then there is a natural choice of $\gamma$ in
\eqref{mms2002.87}. This corresponds to the `natural' smooth Azumaya bundle
associated to the given decomposition of $\delta =\alpha \cup\beta$ and we
take $\Kt^0(X;\alpha ,\beta)$ in \eqref{mms2002.197} to be defined by this
Azumaya bundle, discussed as a warm-up exercise in
Section~\ref{Circlecase}. In Appendix~\ref{appendix:universal} it is
observed that any fibration for which $\beta$ is a multiple of a degree $2$
characteristic class of $\phi:Y\longrightarrow X$ satisfies the hypothesis
in \eqref{mms2002.59}.

In general, the data \eqref{mms2002.115} -- \eqref{mms2002.87} are shown
below to determine an infinite rank `smooth Azumaya bundle', which we
denote $\smA(\gamma).$ This has fibres isomorphic to the algebra of
smoothing operators on the fibre, $Z,$ of $Y$ with Schwartz kernels
consisting of the smooth sections of a line bundle $J(\gamma)$ over $Z^2.$
The completion of this algebra of `smoothing operators' to a bundle with
fibres modelled on the compact operators has Dixmier-Douady invariant
$\alpha\cup\beta.$

In outline the construction of $\smA(\gamma)$ proceeds as follows; details
may be found in Section~\ref{SmoothAzbun}. The trivialization
\eqref{mms2002.87} induces a groupoid character $Y^{[2]}\longrightarrow
\UU(1),$ where $Y^{[2]}$ is the fiber product of two copies of
fibration. Combined with the choice \eqref{mms2002.115} this gives a map
from $Y^{[2]}$ into the torus and hence by pull-back the line bundle
$J=J(\gamma).$ This line bundle is \emph{primitive} in the sense that under
lifting by the three projection maps
\begin{equation}
\xymatrix{\tilde L^{[3]}\ar@<2ex>[r]^{\pi_S}
\ar[r]^{\pi_C}
\ar@<-2ex>[r]^{\pi_F}&\tilde L^{[2]}}
\label{mms2002.138}\end{equation}
(corresponding respectively to the left two, the outer two and the right
two factors) there is a natural isomorphism 
\begin{equation}
\pi_S^*J\otimes\pi_F^*J=\pi_C^*J.
\label{mms2002.139}\end{equation}
This is enough to give the space of global sections,
$\CI(Y^{[2]};J\otimes\Omega _R),$ where $\Omega _R$ is the fiber-density
bundle on the right factor, a fibrewise product isomorphic to the smoothing
operators on $Z.$ Indeed, if $z$ represents a fiber variable
then
\begin{equation}
A\circ B(x,z,z')=
\int_{Z}A(x,z,z'')\cdot B(x,z'',z')
\label{mms2002.190}\end{equation}
where $\cdot$ denotes the isomorphism \eqref{mms2002.139} which gives the
identification 
\begin{equation}
J_{(z,z'')}\otimes J_{(z'',z')}\simeq J_{(z,z')}
\label{mms2002.191}\end{equation}
needed to interpret the integral in \eqref{mms2002.190}. The naturality of
the isomorphism corresponds to the associativity of this product.

Then the smooth Azumaya bundle is defined in terms of its space of global
sections
\begin{equation}
\CI(X;\smA(\gamma))=\CI(Y^{[2]};J(\gamma )).
\label{mms2002.200}\end{equation}
As remarked above, $J(\gamma),$ and hence also the Azumaya bundle,
depends on the particular global trivialization \eqref{mms2002.87}. Two
trivializations, $\gamma _i,$ $i=1,2$ as in \eqref{mms2002.87} determine
\begin{equation}
\gamma _{12}:Y\longrightarrow \UU(1),\ \gamma _{12}(y)\gamma_2(y)=\gamma_1(y)
\label{mms2002.199}\end{equation}
which fixes an element $[\gamma_{12}]\in \operatorname{H}^1(Y;\bbZ)$ and hence a line
bundle $K_{12}$ over $Y$ with Chern class $[\gamma_{12}]\cup
[\phi^*\alpha].$ Then 
\begin{equation}
J(\gamma _2)\simeq (K_{12}^{-1}\boxtimes K_{12})\otimes J(\gamma _1)
\label{mms2002.201}\end{equation}
with the isomorphism consistent with primitivity.

Pulling back to $Y,$ $\phi^*\cA(\gamma)$ is trivialized as an Azumaya bundle
and this trivialization induces an isomorphism of twisted and untwisted
K-theory 
\begin{equation}
\Kt^0(Y;\phi^*\cA(\gamma))\overset{\simeq}\longrightarrow\Kt^0(Y).
\label{mms2002.192}\end{equation}
In fact there are stable isomorphisms between the different smooth Azumaya
bundles and these induce natural and consistent isomorphisms 
\begin{equation}
\Kt^0(X;\cA(\gamma ))\overset{\simeq}\longrightarrow \Kt^0(X;\alpha,\beta).
\label{mms2002.203}\end{equation}
The proof may be found in Section~\ref{Stabiso}.

The transition maps for the local presentation of the smooth Azumaya
bundle, $\smA(\gamma),$ determined by the data \eqref{mms2002.115} --
\eqref{mms2002.87}, are given by multiplication by smooth functions. Thus
they also preserve the corresponding spaces of differential, or
pseudodifferential, operators on the fibres; the corresponding algebras of
twisted fibrewise pseudodifferential operators are therefore well
defined. Moreover, since the principal symbol of a pseudodifferential
operator is invariant under conjugation by (non-vanishing) functions there
is a well-defined symbol map from the pseudodifferential extension of the
Azumaya bundle, with values in the usual symbol space on $T^*(Y/X)$ (so
with no additional twisting). The trivialization of the Azumaya bundle over
$Y,$ and hence over $T^*(Y/X),$ means that the class of an elliptic element
can also be interpreted as an element of
$\Kc^0(T^*(Y/X);\rho^*\phi^*\cA(\gamma))$ where $\rho:T(Y/X)\longrightarrow
Y$ is the bundle projection. This leads to the analytic index map,
\begin{equation}
\ind_{\text{a}}:
\Kc^0(T^*(Y/X);\rho^*\phi^*\cA(\gamma))\longrightarrow\Kt^0(X;\cA(\gamma)).
\label{mms2002.55}\end{equation}

The topological index can be defined using the standard argument by
embedding of the fibration $Y$ into the product fibration $\pi:\bbR^N\times
X\longrightarrow X$ for large $N.$ Namely, the Azumaya bundle is
trivialized over $Y$ and this trivialization extends naturally to a fibred
collar neighborhood $\Omega$ of $Y$ embedded in $\bbR^N\times X.$ Thus, the
usual Thom map $\Kc^0(T^*(Y/X))\longrightarrow \Kc^0(T^*(\Omega /X)$ is
trivially lifted to a map for the twisted K-theory, which then extends by
excision to a map giving the topological index as the composite with Bott
periodicity:-
\begin{multline}
\ind_{\text{t}}:
\Kc^0(T^*(Y/X);\rho^*\phi^*\cA(\gamma))\longrightarrow
\Kc^0(T^*(\Omega/X);\rho^*\tilde\pi^*\cA(\gamma))\\
\longrightarrow\Kc^0(T^*(\bbR^N/X);\rho^*\pi^*\cA(\gamma))\longrightarrow
\Kt^0(X;\cA(\gamma)).
\label{mms2002.196}\end{multline}

In the proof of the equality of these two index maps we pass through an
intermediate step using an index map given by semiclassical quantization of
smoothing operators, rather than standard pseudodifferential
quantiztion. This has the virtue of circumventing the usual problems with
multiplicativity of the analytic index even though it is somewhat less
familiar. A fuller treatment of this semiclassical approach can be found in
\cite{fmtga} so only the novelties, such as they are, in the twisted case
are discussed here. The more conventional route, as used in \cite{mms1}, is
still available but is technically more demanding. In particular it is
worth noting that the semiclassical index map, as defined below, is
well-defined even for a general fibration -- without assuming that
$\phi^*\beta=0.$ Indeed, this is essential in the proof, since the product
fibration $\bbR^N\times X$ does not have this property.

For a fixed fibration the index maps induced by two different
trivializations $\gamma$ may be compared and induce a commutative diagram 
\begin{equation}
\xymatrix{
\Kc^0(T^*(Y/X))\ar[r]^-{\simeq}\ar[dd]^{[K_{12}]\times}&
\Kc^0(T^*(Y/X);\rho^*\phi^*\cA(\gamma_1))
\ar[r]^-{\ind(\gamma _1)}&
\Kt^0(X;\cA(\gamma _1))\ar[d]^-{\simeq}
\\
&&\Kt^0(X;\alpha,\beta )
\\
\Kc^0(T^*(Y/X))\ar[r]^-{\simeq}&
\Kc^0(T^*(Y/X);\rho^*\phi^*\cA(\gamma_2))
\ar[r]^-{\ind(\gamma _2)}&
\Kt^0(X;\cA(\gamma _2)).\ar[u]_-{\simeq}&
}
\label{mms2002.204}\end{equation}
This follows from the proof of the index theorem.

The smoothness of the Azumaya bundle here allows us to give an explicit
Chern-Weil formulation for the index in twisted cohomology. 
We recall that the twisted deRham cohomology $\operatorname{H}^*(X;\delta)$ is
obtained by deforming the deRham differential to $d+\delta\wedge,$ where
$$
\delta =\bar\alpha \wedge \displaystyle\frac{\bar\beta}{2\pi i}.
$$
Here $\bar\alpha = u^*(\theta)$ is the closed 1-form on $X$ with integral
periods, where $\theta$ is the Cartan-Maurer 1-form on $\UU(1)$ and
$\bar\beta$ is the closed 2-form with integral periods which is the
curvature of the hermitian connection $\gamma$ on $L$.

We remark that our results easily extend to the case when the Dixmier-Douady 
class is the sum of decomposable classes, ie when it is in the $\bbZ$-span of
$\operatorname{H}^2(X;\bbZ)\cup \operatorname{H}^1(X;\bbZ)$. The Azumya bundle in this case is the 
tensor product of the decomposable Azumaya bundles as defined 
in this paper.
The case of an arbitrary, not necessarily decomposable, Dixmier-Douady
invariant is postponed to a subsequent paper where the twisted index
theorem is treated in full.  The general case uses pseudodifferential
operators valued in the Azumaya bundle, rather than the pseudodifferential
operator extension of the smooth Azumaya bundle as discussed here. Again, the
semiclassical index map extends without difficulty to this general case.

In outline the paper proceeds as follows. The special case of the circle
bundle is discussed in \S\ref{Circlecase} and \S\ref{decompGeom} contains
the geometry of the general decomposable case. The smooth Azumaya bundle
corresponding to a decomposable Dixmier-Douady class is constructed in
\S\ref{SmoothAzbun} and some examples are also given. In \S\ref{Stabiso},
\eqref{mms2002.203} is proved. The analytic index maps is defined in
\S\ref{sect:analytic ind} using spaces of projective elliptic operators but
including the case of twisted families of Dirac operators. The topological
index is defined in \S\ref{sect:top ind}. The Chern-Weil representative of
the twisted Chern Character is studied in
\S\ref{TwistedChern}. Semiclassical versions of the index maps are
introduced in \S\ref{scl} and \S\ref{sect:index thm} contains the proof of
the equality of these two indices. In \S\ref{sect:coh}, the Chern character
of the index is computed. In Appendix~\ref{diffchar} the formulation of the
Dixmier-Douady invariant in terms of differential characters is explored
and in Appendix~\ref{appendix:cech} it is computed using \v Cech cohomology
(following a similar computation by
Brylinski). Appendix~\ref{appendix:universal} contains a discussion of the
conditions on a fibration under which a line bundle from the base is
trivial when lifted to the total space. It also contains the description of 
a canonical projective family of Dirac operators on a Riemann surface.

\paperbody

\section{Trivialization by the circle bundle}\label{Circlecase}

An element $\beta \in \operatorname{H}^2(X;\bbZ)$ for a compact manifold $X,$ represents
an isomorphism class of line bundles over $X.$ Let $L$ be such a line
bundle with Hermitian inner product $h$ and unitary connection $\nabla^L.$
We proceed to outline the construction of the smooth Azumaya bundle in the
special case, alluded to above, where $Y=\tilde L$ is the circle bundle of
$L.$ This is carried out separately since this case gives a \emph{natural}
choice of the smooth Azumaya bundle, and hence the twisted K-group. The
corresponding twisted cohomology is also identified with the cohomology of
a subcomplex of the deRham complex over $\tilde L.$

From $u\in\CI(X;\UU(1))$ construct the principal $\bbZ$-bundle
\begin{equation}
\xymatrix{\bbZ\ar@{-}[r]&\hat X\ar[d]^{\pi}\\
&X}
\label{mms2002.61}\end{equation}
with total space the possible values of $\frac1{2\pi i}\log u$ over points
of $X$ and with \ci\ structure determined by the smoothness of 
a local branch of this function. Thus $f=\frac1{2\pi i}\log u\in\CI(\hat
X;\bbR)$ is a well defined smooth function and under deck transformations
\begin{equation}
f(\hat x+n)=f(\hat x)+n\ \forall\ \hat x\in \hat X,\ n\in\bbZ.
\label{mms2002.62}\end{equation}
Let $p:\tilde L\longrightarrow X$ be the circle bundle of $L;$ pulled back to
$\tilde L,$ $L$ is canonically trivial. If $\nabla^L$ is an Hermitian
connection on $L$ then pulled back to $\tilde L$ it is of the form $d+\gamma$
on the trivialization of $L,$ with $\gamma\in\CI(\tilde L,\Lambda ^1)$ a
principal $\UU(1)$-bundle connection form in the usual sense. That is,
under the action 
\begin{equation}
m:\UU(1)\times\tilde L\longrightarrow \tilde L,
\label{mms2002.144}\end{equation}
$m^*\gamma=id\theta +\gamma.$ This corresponds to the `fiber
shift map' on the fiber product 
\begin{equation}
s:\tilde L^{[2]}=
\tilde L\times_X\tilde L\longrightarrow \UU(1),\ \tilde l_1=s(\tilde
l_1,\tilde l_2)\tilde l_2\Min \tilde L_x
\label{mms2002.89}\end{equation}
in that $d\log s=p_1^*\gamma-p_2^*\gamma$ is the difference of
the pull-back of the connection form from the two factors. From the
character $s$ a bundle, $J,$ can be constructed from the trivial
bundle over the fiber product $Q=\tilde L\times_X\tilde L\times _X\hat X$
corresponding to  the identification
\begin{equation}
(\tilde l_1,\tilde l_2,\hat x+n,z)\simeq (\tilde l_1,\tilde l_2,\hat x,
s(\tilde l_1,\tilde l_2)^nz).
\label{mms2002.47}\end{equation}
Thus, $J$ is associated to $Q$ as a principal $\bbZ$-bundle over $\tilde
L^{[2]}.$ The primitivity property \eqref{mms2002.139} follows from the
multiplicativity property of $s,$ that $s(l_1,l_2)s(l_2,l_3)=s(l_1,l_2)$
for any three points in a fixed fiber, which in turn follows from
\eqref{mms2002.89}. The connection $d+fd\log s$ on the trivial bundle over
$Q$ descends to a connection on $J$ which has curvature equal to a
difference 
\begin{equation}
\omega_J=\frac1{2\pi i}
\overline{\alpha}\wedge d\log s=\overline{\alpha}\wedge
\frac1{2\pi i}(p_1^*\gamma-p_2^*\gamma ),\
\overline{\alpha}=df.
\label{mms2002.145}\end{equation}
By definition, the space of global sections of the smooth Azumaya bundle is 
\begin{equation}
\CI(X;\cA)=\CI(\tilde L^{[2]};J),
\label{29.5.2008.1}\end{equation}
where the product on the right hand side is given by composition of
Schwartz kernels.

The `Dixmier-Douady twisting', given the decomposed form, corresponds to
two different trivializations of $J.$ Over any open set $U\subset X$ where
$u$ has a smooth logarithm, $J$ is trivial using the section of $\hat X$
this gives. On the other hand, over any open set $U\subset X$ over which
$\tilde L$ has a smooth section $\tau,$ the character in \eqref{mms2002.47}
is decomposed as the product $s(\tilde l_1,\tilde l_2)=s_{\tau}(\tilde
l_1)s_{\tau}(\tilde l_2)^{-1}$ where $s_{\tau}(\tilde l)=s(\tilde
l,\tau(p(\tilde l)).$ This allows a line bundle $K$ to be defined over the
preimage of $U$ in $\tilde L$ by the identification of the trivial bundle
\begin{equation}
(\tilde l,\hat x+n,z)\simeq(\tilde l,\hat x,s_{\tau}(\tilde l)^nz).
\label{mms2002.177}\end{equation}
Clearly then $J$ may be identified with $K\boxtimes K',$ where $K'$ is
the dual, over $U.$ In terms of a local trivialization in both senses over
a small open set $U\subset X,$ in which $L_U=U\times\bbC,$ $\tilde
L_U=U\times\bbS,$ $\hat X_U=U\times\bbZ,$ $a(x,\theta ,\theta ')\in
\CI(U\times\bbZ\times\bbS\times\bbS)$ satisfies
\begin{equation}
a(x,n,\theta,\theta')=e^{in\theta}a(x,0,\theta,\theta')e^{-in\theta'}.
\label{29.5.2008.2}\end{equation}
This twisted conjugation means that $\cA$ is a bundle of algebras, modelled
on the smoothing operators on the circle with \eqref{29.5.2008.2} giving
local algebra trivializations. In this case the Azumaya bundle is
associated with the principal $\UU(1)\times\bbZ$ bundle $\tilde
L\times_X\hat X$ and to the projective representation of this structure group
through its central extension to the Heisenberg group.

The corresponding construction in the general case is quite similar and is
described in \S\ref{SmoothAzbun}.

The 3-twisted cohomology on $X,$ with twisting form $\overline{\delta}
=\overline{\alpha} \wedge\overline{\beta},$ is the target for the twisted
Chern character discussed below. Here $\overline{\alpha}$ is a closed
1-form and $\overline{\beta}$ is the curvature 2-form on $X$ for the
Hermitian connection on $L.$ Thus, on $\tilde L,$ $d\gamma=(2\pi
i)p^*\overline{\beta}.$ In fact the $\overline{\delta}$-twisted deRham
cohomology on $X$ can be expressed as the cohomology of a subcomplex of the
ordinary (total) deRham complex on $\tilde L.$

\begin{proposition}\label{mms2002.146} The even and odd degree subspaces of
  $\CI(\tilde L,\Lambda ^*)$ fixed by the conditions with respect to the
  infinitesmal generator of the $\UU(1)$ action on $\tilde L$
\begin{equation}
\cL_{\pa/\pa \theta }\tilde v=0,\
\iota _{\pa/\pa\theta}\tilde v=\frac{p^*\overline{\alpha}}{2\pi}
\wedge\tilde v,\ \tilde
v\in\CI(\tilde L,\Lambda ^*)
\label{mms2002.147}\end{equation}
are mapped into each other by the standard deRham differential which has
cohomology groups canonically isomorphic to the $\overline{\delta}$-twisted
deRham cohomology on $X.$
\end{proposition}

\begin{proof} The conditions in \eqref{mms2002.147} are preserved by $d$
since it commutes with the Lie derivative and given the first condition
\begin{equation}
\iota _{\pa/\pa\theta }d\tilde v=\cL_{\pa/\pa \theta }\tilde
v-d(\frac{p^*\overline{\alpha}}{2\pi}\wedge\tilde v)=
\frac{p^*\overline{\alpha}}{2\pi}\wedge d\tilde v.
\label{mms2002.148}\end{equation}

If $\tilde v$ satisfies \eqref{mms2002.147} then $v'=\tilde
v-\frac{\gamma}{2\pi i}\wedge p^*\overline{\alpha}\wedge\tilde v$ satisfies 
\begin{equation}
\cL_{\pa/\pa \theta } v'=0,\ \iota _{\pa/\pa\theta} v'=0\Longrightarrow
v'=p^*v,\ v\in\CI(X;\Lambda ^*).
\label{mms2002.149}\end{equation}
Conversely if $v\in\CI(X;\Lambda ^*)$ then $\tilde
v=p^*v+\frac{\gamma}{2\pi i}
\wedge\overline{\alpha}\wedge p^*v$ satisfies \eqref{mms2002.147}. Thus every form
satisfying \eqref{mms2002.147} can be written uniquely 
\begin{equation}
\tilde v=\exp\left(\frac{\gamma\wedge p^*\overline{\alpha}}{2\pi i}\right)p^*v=
p^*v+\frac{\gamma \wedge p^*\overline{\alpha}}{2\pi i}\wedge p^*v.
\label{mms2002.150}\end{equation}
Under this isomorphism $d$ is clearly conjugated to $d+\overline{\delta}\wedge$
proving the Proposition.
\end{proof}

\section{Geometry of the decomposed class}
\label{decompGeom}

For a given line bundle $L$ over $X$ consider a fiber bundle
\eqref{mms2002.59} such that $L$ is trivial when lifted to the total space.
As discussed above, the circle bundle $\tilde L$ is an example. A more
general discussion of this condition can be found in Appendix
\ref{appendix:universal}. An explicit trivialization of the lift, $\gamma,$
as in \eqref{mms2002.87} is equivalent to a global section which is the
preimage of $1:$
\begin{equation}
s':Y \longrightarrow \phi^*(\tilde L).
\label{mms2002.56}\end{equation}

Over each fiber of $Y,$ the image is fixed so this determines a map
$$
s(z_1,z_2)=s'(z_1)(s'(z_2))^{-1}
$$
which is well-defined on the fiber product and is a groupoid character:
\begin{equation}
\begin{gathered}
s:Y^{[2]}\longrightarrow
\UU(1),\\
s(z_1,z_2)s(z_2,z_3)=s(z_1,z_3)\ \forall\ z_i\in Y\Mwith
\phi(z_i)=x,\ i=1,2,3,\ \forall\ x\in X.
\end{gathered}
\label{mms2002.57}\end{equation}
Conversely one can start with a unitary character $s$ of $Y^{[2]}$ and
recover $L$ as the associated Hermitian line bundle
\begin{equation}
\begin{gathered}
L=Y\times\bbC/\simeq_s,\\
(z_1,t)\simeq_s(z_2,s(z_2,z_1)t)\ \forall\ t\in\bbC,\ \phi(z_1)= \phi(z_2).
\end{gathered}
\label{mms2002.58}\end{equation}
The connection on $L$ lifts to a connection 
\begin{equation}
\phi^*\nabla^L=d+\gamma ,\quad \gamma\in\CI(Y;\Lambda ^1),\quad
\pi_1^*\gamma -\pi_2^*\gamma =d\log s\Mon Y^{[2]}
\label{mms2002.60}\end{equation}
on the trivial bundle $\phi^*(L).$ Conversely any 1-form on $Y$ with this
property defines a connection on $L.$

Now, let $Q=Y^{[2]}\times_X\hat X$ be the fiber product of $Y^{[2]}$ and
$\hat X,$ so as a bundle over $X$ it has typical fiber $Z^2\times\bbZ;$ it
is also a principal $\bbZ$-bundle over $Y^{[2]}.$ The data above determines
an action of $\bbZ$ on the trivial bundle $Q\times\bbC$ over $Q,$
namely  
\begin{equation}
T_n:(z_1,z_2,\hat x;w)\longrightarrow (z_1,z_2,\hat x+n,s(z_1,z_2)^{-n}w)\
\forall\ n\in \bbZ.
\label{mms2002.63}\end{equation}
Let $J$ be the associated line bundle over $Y^{[2]}$ 
\begin{equation}
J =(Q\times\bbC)/\simeq,\quad (z_1,z_2,\hat x;w)\simeq T_n(z_1,z_2,\hat x;w)\
\forall\ n\in \bbZ.
\end{equation}
The fiber of $J$ at $(z_1, z_2) \in Y^{[2]}$ such that $\phi(z_1)=\phi(z_2)=x$ is 
\begin{equation}
J_{z_1,z_2}=\hat X_x\times\bbC/\simeq,\quad (\hat x+n,w)
\simeq (\hat x,s(z_1,z_2)^nw).
\label{mms2002.74}\end{equation}

\begin{lemma}\label{mms2002.64} The connection $d+fd\log s$
on $Q\times\bbC$ descends to a connection $\nabla^J$ on $J$ which has curvature 
\begin{equation}
F _{\nabla^J}=\pi^*_1\mu-\pi^*_2\mu,\
\mu=df\wedge\frac{\gamma}{2\pi i}\in\CI(Y;\Lambda ^2),\
\xymatrix@1{Y^{[2]}\ar@<1ex>[r]^{\pi_1}\ar[r]_{\pi_2}&Y.}
\label{mms2002.65}\end{equation}
Moreover $d\mu=\phi^*(\overline{\delta}),$ for the uniquely determined
3-form on $X,$ $\overline{\delta} = \overline{\alpha} \wedge
\overline{\beta} \in \CI(X;\Lambda ^3),$ where $df =
\phi^*(\overline{\alpha})$ and $d\gamma = 2\pi i\phi^*(\overline{\beta})$
represent the characteristic class of $\hat X$ and the first Chern class
of $L$ respectively.
\end{lemma}

\begin{proof} Clearly the 1-form $fd\log s$ has the correct transformation
  law under the $\bbZ$ action on $Y^{[2]}\times_X\hat X$ to 
  give a connection on $J.$ Its curvature is $\overline{\alpha}\wedge \frac{d\log
  s}{2\pi i}$ where $\overline{\alpha}=\frac{1}{2\pi i}d\log u.$ If $\gamma$
  is the connection form for the trivialization of $L$ on $Y$ then  
\begin{equation}
d\log s=\pi_1^*\gamma -\pi_2^*\gamma\Min Y^{[2]}
\label{mms2002.143}\end{equation}
from which \eqref{mms2002.65}, together with the remainder of the Lemma, follows.
\end{proof}

\section{Smooth Azumaya bundle}\label{SmoothAzbun}

We proceed to show how to associate to the data \eqref{mms2002.115} --
\eqref{mms2002.87} discussed above a smooth Azumaya bundle over $X.$ That
is, we construct a locally trivial bundle with fibres modelled on the
smoothing operators on the sections of a line bundle over the fibres of $Y$
and having completion with Dixmier-Douady invariant $\alpha\cup\beta.$ Note
that this Azumaya bundle \emph{does depend} on the trivialization data in
\eqref{mms2002.87}; we will therefore denote it $J(\gamma).$ The effect of
changing this trivialization is discussed in Lemma~\ref{TrivChange} below.

First consider local trivializations of the data.

\begin{proposition}\label{mms2002.66} A section of $\phi,$ over on open set
  $U\subset X,$ $\tau:U\longrightarrow \phi^{-1}(U),$ induces a
  trivialization of $L$ over $U$ and an isomorphism of $J(\gamma)$ over the open
  subset $V = \phi^{-1}(U) \times_U \phi^{-1}(U)$ of $Y^{[2]},$ with
\begin{equation}
J\big|_V\cong_{\tau}\Hom(K_{\tau})=K_{\tau}\boxtimes K_{\tau}'
\label{mms2002.68}\end{equation}
for a line bundle $K_\tau$ over $\phi^{-1}(U)\subset Y,$ where $K_\tau'$
denotes the line bundle dual to $K_\tau.$ Another choice of section
$\tau':U\longrightarrow \phi^{-1}(U),$ determines another line bundle
$K_{\tau'}$ over $\phi^{-1}(U)\subset Y,$ satisfying
\begin{equation}
K_{\tau} = K_{\tau'} \otimes \phi^*(L_{\tau, \tau'}),
\label{mms2002.68'}\end{equation}
where $L_{\tau, \tau'}= (\tau, \tau')^*J$ is the fixed local line bundle
over $U.$
\end{proposition}

\begin{proof} A local section of $\phi$ induces a local trivialization of
the character $s,$  
\begin{equation}
s(z_1,z_2)=\chi_{\tau}(z_1)\chi_{\tau}^{-1}(z_2),\quad
\chi_{\tau}(z)=s(z,\tau(\phi(z)))\Mon\phi^{-1}(U)\subset Y.
\label{mms2002.67}\end{equation}
This trivializes $L$ over $U,$ identifying it with $\tau^*\bbC$ with
connection $d+\tau^*\gamma.$

The line bundle $K_{\tau}$ over $\phi^{-1}(U)$ associated to the $\bbZ$
bundle $\phi^{-1}(U)\times_U\hat X_U$ by the identification $(z,\hat
x+n,w)\simeq (z,\hat x,\chi_{\tau}(z)^nw)$ then satisfies \eqref{mms2002.68}.
The line bundle $K_{\tau'}$ is similarly defined over $\phi^{-1}(U)$,  satisfying 
\eqref{mms2002.68} with $\tau'$ substituted for $\tau$. The relation
\eqref{mms2002.68'} follows from \eqref{mms2002.68} and its modification
with $\tau'$ substituted for $\tau$.
\end{proof}

Such a section of $Y$ will induce a local trivialization of the smooth
Azumaya bundle in which it becomes the smoothing operators on the fibres of
$Y$ acting on sections of $K_{\tau}:$
\begin{equation}
\smA_{\tau}=\Psi^{-\infty}(\phi^{-1}(U)/U;K_{\tau}).
\label{mms2002.69}\end{equation}
Using Proposition \ref{mms2002.66}, we get the local patching,
\begin{equation}
\smA_{\tau} = \smA_{\tau'} \otimes \Psi^{-\infty}(\phi^{-1}(U)/U;
\phi^*(L_{\tau, \tau'})).
\label{mms2002.69'}\end{equation}

Rather than use this as a definition we adopt an \emph{a priori} global
definition by trivializing over $Y.$

\begin{definition}\label{def:Azumaya} For any $x\in X,$ the fiber of the
smooth Azumaya bundle associated to the geometric data in \S\ref{decompGeom} is 
\begin{equation}
\smA_x=\CI(Y_x^2;J\big|_{Y_x^2}\otimes\Omega_R)
\label{mms2002.70}\end{equation}
where $\Omega_R$ is the fiber density bundle on the right factor of $Y_x^2.$
Globally, we have a natural identification, 
\begin{equation}
\CI(X,\smA) \cong \CI(Y^{[2]}, J\otimes\Omega_R). 
\label{mms2002.158}\end{equation}
\end{definition}
\noindent Thus a smooth section of $\smA$ over any open set $U\subset X$ is
just a smooth section of $J\otimes\Omega _R,$ where
$\Omega_R=\pi_R^*\Omega,$ over the preimage of $U$ in $Y^{[2]}.$ 
  
Of course, we need to show that $\smA$ is a bundle of algebras over $X$ with
local trivializations as indicated in \eqref{mms2002.69}. To see this
globally, observe that $J$ has the same `primitivity' property as for
$\tilde L$ in \S\ref{Circlecase} with respect to the groupoid structure.

\begin{lemma}\label{mms2002.71} If 
\begin{equation}
\xymatrix@1{
Y^{[3]}\ar@<2ex>[r]^{\pi_S}
\ar[r]^{\pi_C}
\ar@<-2ex>[r]^{\pi_F}
&Y^{[2]}
}
\label{mms2002.72}\end{equation}
are the three projections (respectively onto the two left-most, the outer
two and the two right-most factors -- the notation stands for `second',
`composite' and `first' for operator composition) then there is a natural
isomorphism
\begin{equation}
\pi_S^*J\otimes\pi_F^*J  \overset{\simeq}\longrightarrow \pi_C^*J
\label{mms2002.73}\end{equation}
and moreover $J$ carries a connection $\nabla^J$ which respects this
primitivity property.
\end{lemma}

\begin{proof} 
The identity \eqref{mms2002.73} is evident from the definition of $J$ and
Proposition \ref{mms2002.66}. The naturality property for
\eqref{mms2002.73} corresponds to an identity on $Y^{[4]}.$ Namely if $J'$
is the dual of $J$ then the tensor product of the pull-backs under the four
projections $Y^{[4]}\longrightarrow Y^{[3]}$ of the combination
$\pi_S^*J\otimes\pi_F^*J\otimes\pi_C^*J'$ over $Y^{[3]}$ is naturally
trivial. That this trivialization is equal to the tensor product of the
four trivializations from \eqref{mms2002.73} follows again from the
definition of $J.$

By Proposition \ref{mms2002.66}, a section $\tau:U\longrightarrow Y$ of
$\phi$ over the open subset $U$ of $X$ defines an isomorphism
$J\big|_V\cong_{\tau}\Hom(K_{\tau})=K_{\tau}\boxtimes K_{\tau}'$ where $V = 
\phi^{-1}(U) \times_U \phi^{-1}(U)$ is the open subset of $Y^{[2]}.$ A
choice of connection $\nabla^\tau$ on $K_{\tau}$ induces a connection
$\nabla^V$ on $J\big|_V$ which clearly respects the primitivity
property. A global connection preserving the primitivity property can then
be constructed using a partition of unity on $X.$
\end{proof}

As a consequence of Lemma~\ref{mms2002.71} there is a lifting map
\begin{equation}
\CI(Y^{[2]};J\otimes\Omega_R)\overset{\pi_F^*}
\longrightarrow\CI(Y^{[3]};\pi_SJ'\otimes \pi_CJ\otimes\pi_F^*\Omega_{R})
\overset{\simeq}\longrightarrow\Psi^{-\infty}(Y^{[2]}/Y;J') 
\label{mms2002.75}\end{equation}
which embeds into an algebra, namely the smoothing operators on sections
of $J'$ on the fibres of $Y^{[2]}$ as a fibration over $Y$ (projecting onto
the first factor).

\begin{proposition}\label{mms2002.76} Lifting
  $\CI(Y^{[2]};J\otimes\Omega_R)$ to $Y^{[3]}$ under the
  projection off the left-most factor (the `first' projection in terms of
  composition) embeds it as a subalgebra of the smoothing operators on
  sections of $J'$ as a bundle over $Y^{[2]}$ on the fibres of the
  projection onto the right factor such that the lift of the bundle of
  algebras over $X$ is equal to the bundle of algebras over $Y.$
\end{proposition}
\noindent This justifies \eqref{mms2002.69}. As discussed below it also
shows that, as an Azumaya bundle, the completion of $\smA$ is $\cA
=\cA(\gamma)$.

\begin{proof} It only remains to show that composition of two local
  sections of $\smA$ in the algebra of fiber smoothing operators gives
  another section of the Azumaya bundle. However, this follows from
  \eqref{mms2002.69}, which in turn is a consequence of Lemma~\ref{mms2002.71}
  applied to the local decomposition of $J$ in \eqref{mms2002.68}.
\end{proof}

An infinite rank Azumaya bundle $\cA$, over a topological space $X,$ is a
bundle of star algebras with local isomorphisms with the trivial bundle of
compact operators, $\cK(\cH),$ on a fixed separable but infinite-dimensional
Hilbert space $\cH.$ The Dixmier-Douady invariant of $\cA$ is an element
of $\operatorname{H}^3(X;\bbZ).$ It classifies the bundle up to stable isomorphism
(i\@.e\@.~after tensoring with $\cK)$ and can be realized in terms of \v
Cech cohomology or alternatively in terms of classifying spaces as
follows. The group of $*$-automorphism of $\cK$ is
$\PU(\cH)=\UU(\cH)/\UU(1),$ the projective unitary group of the Hilbert
space acting by conjugation. Thus the fiber trivializations of $\cA$ form a
principal $\PU(\cH)$-bundle over $X.$ Since $\PU(\cH)=K(\bbZ,2)$ is an
Eilenberg-Maclane space, this bundle, and hence $\cA,$ is classified up to
isomorphism by an homotopy class of maps $X\longrightarrow
B\PU(\cH)=K(\bbZ,3)$ which represents, and is equivalent to, the
Dixmier-Douady invariant.

The Chern class of a line bundle $L$ over a space $X$ has a similar
representation. Taking an Hermitian structure and passing to the associated
circle bundle $\tilde L$ over $X$ one can consider the Hilbert bundle
$L^2(\tilde L/X)$ of Lebesgue square integrable functions on the fibres of
the circle bundle. Each point $l\in\tilde L$ defines a unitary operator on
the fiber through that point, namely multiplication by $U(\hat
l)=\exp(i\theta_{\hat l})\times$ where the normalization is such that
$\exp(i\theta_{\hat l})(\hat l)=1.$ Changing $\hat l$ within the fiber
changes $U(\hat l)$ to $\exp(i\theta')U(\hat l')$ so this defines a map 
\begin{equation}
X\longrightarrow \PU(L^2(\tilde L/X))
\label{mms2002.116}\end{equation}
into the bundle of projective unitary operators on the fibres of the
Hilbert bundle. By Kuiper's theorem any Hilbert bundle is trivial (in the
uniform topology) and the trivialization is natural up to homotopy. Thus
the map \eqref{mms2002.116} lifts to a unique homotopy class of maps 
\begin{equation}
X\longrightarrow \PU(\cH)=K(\bbZ,2)
\label{mms2002.117}\end{equation}
and this represents, and is equivalent to, the first Chern class. This
follows from the evident fact that $\tilde L$ is isomorphic to the
pull-back of the canonical circle bunde, $\UU(\cH)/\PU(\cH)$ over $\PU(\cH).$

Now consider the decomposed case under consideration here. Over the given
space $X$ we have both a map $u\in\CI(X;\UU(1))$ and a line bundle $L.$
Passing to the classifying map \eqref{mms2002.116} this gives a unique
homotopy class of maps 
\begin{equation}
X\longrightarrow \UU(1)\times\PU(\cH).
\label{mms2002.118}\end{equation}

\begin{proposition}\label{mms2002.119} The completion of the smooth Azumaya
  bundle $\smA$ associated above to \eqref{mms2002.115} --
  \eqref{mms2002.87} to an Azumaya bundle $\cA=\cA(\gamma),$ has
  Dixmier-Douady invariant $\alpha \cup\beta \in \operatorname{H}^3(X;\bbZ)$ which is
  represented by the composite of \eqref{mms2002.118} with the classifying
  map $\UU(1)\times\PU(\cH)\longrightarrow K(\bbZ,3)$ induced by the
  projectivisation of the basic representation of the Heisenberg group
  $\bbZ\times\UU(1)\longrightarrow \PU(\cH).$
\end{proposition}

\begin{proof} The classifying space $BG$ of a topological group $G$ is
  defined up to homotopy as the quotient $*/G$ of a contractible space on
  which $G$ acts freely. In particular it follows that (always up to
  homotopy) 
\begin{equation}
B(G_1\times G_2)\simeq BG_1\times BG_2
\label{mms2002.120}\end{equation}
and if $H\subset G$ is a closed subgroup then there is a well defined
homotopy class of maps 
\begin{equation}
BH\longrightarrow BG.
\label{mms2002.121}\end{equation}

Recall that the basic representation of the Heisenberg group $H$ arises
from the actions of $\UU(1)$ and $\bbZ$ on $L^2(\bbS)$ (or $\CI(\bbS))$
respectively by translation and multiplication by $e^{in\theta}.$ These
commute up to scalars, which is the action of the center of $H$ as
a central extension  
\begin{equation}
\UU(1)\longrightarrow H\longrightarrow \bbZ\times\UU(1)
\label{mms2002.122}\end{equation}
and so embeds 
\begin{equation}
\bbZ\times \UU(1)\hookrightarrow \PU(\cH)
\label{mms2002.123}\end{equation}
as a subgroup of the projective unitary group on $L^2(\bbS).$ By
\eqref{mms2002.120} and \eqref{mms2002.121} this induces an homotopy class
of continuous maps 
\begin{equation}
\Delta :\UU(1)\times\PU\simeq B(\bbZ\times\UU(1))\longrightarrow K(\bbZ,3).
\label{mms2002.124}\end{equation}
So the claim in the Proposition is that under this map the pull-back of the
degree 3 generator of the cohomology of $K(\bbZ,3)$ is the Dixmier-Douady invariant of $\cA$
and is equal to $\alpha \cup\beta$ in $\operatorname{H}^3(X,\bbZ).$

The first statement follows from the fact that the $\PU$ bundle to which
$\cA$ is associated is obtained from the $\bbZ\times\UU(1)$ bundle $\hat
X\times_X\tilde L$ by extending the structure group using
\eqref{mms2002.123}. The second statement follows from the fact that under
the map \eqref{mms2002.124} the generating 3-class
$\delta_{\operatorname{DD}} \in \operatorname{H}^3(K(\bbZ,3), \bbZ)$ pulls back to $\alpha '\cup\beta
'$ where $\alpha'\in \operatorname{H}^1(\bbS,\bbZ)$ and $\beta '\in \operatorname{H}^2(\PU,\bbZ)$ are the
generators, that is,
\begin{equation}
\Delta ^*\delta =\alpha '\cup\beta '.
\label{mms2002.125}\end{equation}
Indeed, the degree 3-cohomology of $\UU(1)\times\PU$ has a 
single generator, so
\eqref{mms2002.125} must be correct up to a multiple on the right
side. Thus it is enough to check one example, to determine that the multiple 
is equal to one. Take $X=\bbS\times\bbS^2$
with $u$ the identity on $\bbS$ and $L$ the standard line bundle over the
sphere. We know that the induced map \eqref{mms2002.117} for the sphere
generates the second homotopy group of $\PU$ and pulls back to the
fundamental class on $\bbS^2.$ Thus it suffices to note that the $\PU$
bundle over $\bbS\times\bbS^2$ with which our smooth Azumaya bundle is
associated in this case is just obtained by the clutching construction from
the trivial bundle over $[0,2\pi]\times\bbS^2$ using this map.
\end{proof}

An interesting special case of this construction, close to the lifting to
the circle bundle described in \S\ref{Circlecase}, arises when $\beta\in
\operatorname{H}^2(X;\bbZ)$ is thought of as the first Chern class of a complex vector
bundle rather than a line bundle. Then $Y$ can be taken to be the
associated principal bundle 
\begin{equation*}
\xymatrix@1{\UU(n)\ar@{-}[r]&P\ar[d]\\
&X.}
\label{mms2002.91}\end{equation*}

Since the abelianization of $\UU(n)$ is canonically isomorphic to $\UU(1),$
any character (i\@.e\@.~1-dimensional unitary representation) of $\UU(n)$
factorizes through $\UU(1)$, and conversely, any character of $\UU(1)$
lifts to a character of $\UU(n).$ A $\UU(1)$-central extension of the group
$\UU(n) \times \bbZ$ arises in the form of a generalized Heisenberg
group. Namely the group product on $H_n=\UU(n) \times \bbZ \times \UU(1)$
can be taken to be
$$
(g_1, n_1, z_1)(g_2, n_2, z_2) = (g_1g_2, n_1 + n_2,
\det(g_1)^{n_2} z_1 z_2).
$$
Then
$$
1\longrightarrow \UU(1)\longrightarrow  H_n\longrightarrow \UU(n)\times
\bbZ\longrightarrow1
$$
is a central extension.

\section{Stable Azumaya isomorphism}\label{Stabiso}

We proceed to show that the twisted K-groups, $\Kt^0(X;\cA(\gamma )),$
defined through the possible data \eqref{mms2002.115} -- \eqref{mms2002.87}
corresponding to a fixed decomposition \eqref{mms2002.137} are all
naturally isomorphic, as indicated in \eqref{mms2002.203}. This is a
consequence of the Morita invariance of the $C^*$ K-groups and the existence
of stabilized isomorphisms between the various Azumaya bundles.

For a smooth 1-parameter family of trivializations, as in
\eqref{mms2002.87}, so depending smoothly on $t\in[0,1],$ the K-groups
$\Kt^0(X;\cA(\gamma (t)))$ are all naturally isomorphic. Since two such
trivializations differ by a smooth map $Y\longrightarrow \UU(1),$ the
K-group can only depend on the homotopy class of this map, when the other
data is fixed. It is also the case that K-theory of $C^*$ algebras admits
Morita equivalence. That is, the K-group of $\cA$ is naturally isomorphic
to the K-group of $\cA\otimes\cK.$ Kuiper's theorem shows that the
completion of the smoothing operators on any fiber bundle over a space, and
acting on sections of any vector bundle, $V,$ over the fiber bundle, is
naturally isomorphic up to homotopy to the trivial Azumaya bundle $\cK.$ It
follows that the `twisted' K-theory of a space, computed with respect to
such a bundle is naturally isomorphic to the untwisted K-theory. More
generally, taking the smooth Azumaya bundle $\smA(\gamma)$ and tensoring
with the bundle of smoothing operators, $\Psi^{-\infty}(\psi;V),$ on any
other fiber bundle $\psi:Y'\longrightarrow X,$ over the same base, gives an
Azumaya bundle with the same twisted K-theory, $\Kt^0(X;\cA(\gamma)).$ This
proves:-

\begin{lemma}\label{mms2002.206} If $\psi:Y'\longrightarrow X$ is a
  fibration of compact manifolds and $\smA(\gamma)$ is the Azumaya bundle
  associated to data \eqref{mms2002.115} -- \eqref{mms2002.87} then there
  is a natural isomorphism of twisted K-theory 
\begin{equation}
\Kt^0(X;\cA(\gamma ))\overset{\simeq}\longrightarrow \Kt^0(X;\cA(\gamma '))
\label{mms2002.207}\end{equation}
where $\gamma '$ is the trivialization obtained by pulling back $\gamma$
to the product bundle $Y\times_XY'\longrightarrow X.$
\end{lemma}

Applying this result to the initial Azumaya bundle in \S~\ref{Circlecase}
and the general case, shows that $\Kt^0(X;\alpha ,\beta)$ and 
$\Kt^0(X;\cA(\gamma ))$ are each naturally isomorphic to some (possibly
different) $\Kt^0(X;\cA(\gamma'))$ where $\gamma'$ is a trivialization of the
lift of $\tilde L$ to $\tilde L\times_X Y,$ obtained in the two cases by
lifting the trivialization from $\tilde L$ or $Y$ to the fibre
product. Thus it remains to consider two different trivializations over the
same fibration.

\begin{proposition}\label{mms2002.205} If $\gamma_i$ are two
  trivializations of $\phi^*L$ over $Y$ as in \eqref{mms2002.87} then there
  is an embedding of algebras, unique up to homotopy, 
\begin{equation}
\smA(\gamma _2)\longrightarrow \smA(\gamma _1)\boxtimes\Psi^{-\infty}(\bbT^2;K)
\label{mms2002.208}\end{equation}
for a line bundle over the 2-torus which induces natural isomorphisms
\begin{equation}
\Kt^0(X;\cA(\gamma _2))\overset{\simeq}\longrightarrow \Kt^0(X;\cA_{12}))
\overset{\simeq}\longrightarrow \Kt^0(X;\cA(\gamma _1))
\label{mms2002.209}\end{equation}
where $\cA_{12}$ is the completion of
$\smA(\gamma_1)\boxtimes\Psi^{-\infty}(\bbT^2;K).$ 
\end{proposition}

\begin{proof} This is really an adaptation of the proof of the index
  theorem via embedding. First, we recall the discussion above, which shows
 that the primitive line bundle $J(\gamma _2)$ is isomorphic to $J(\gamma
 _1)\otimes(K_{12}\boxtimes K_{12}')$ for a line bundle $K_{12}$ over $Y$
 pulled back from a line bundle $K$ over $\bbT$ by a smooth map
 $\kappa_{12}:Y\longrightarrow \bbR^2.$ This map embeds $Y$ as a
 subfibration of $\phi\circ\pi_1:Y\times\bbT^2\longrightarrow X.$ Let
 $N\longrightarrow Y$ be the normal bundle to this embedding. Given a
 metric this carries a field of harmonic oscillators on the fibres, the
 ground states of which give the desired embedding.

Let $v(z,\zeta)$ be the $L^2$-orthonormalized ground state on the fiber
over $z\in Y.$ Then   
\begin{equation}
\CI(Y^{[2]};J(\gamma ))\ni a(z_1,z_2)\longmapsto
\tilde a=v(z_1,\zeta_1)a(z_1,z_2)v(z_2,\zeta _2)\in\cS(V^{[2]};J(\gamma_2))
\label{mms2002.210}\end{equation}
is an embedding. Moreover, this is an embedding of algebras with the
algebra structure on the right given by Schwartz-smoothing operators. Now
consider the bundle $J(\gamma_1)\boxtimes K\boxtimes K'$ over
$Y^{[2]}\times\bbT^2\times\bbT^2.$ Restricted to the image of the embedding
of $Y^{[2]}$ given by $\kappa_{12}$ acting in both fibres, this is
isomorphic to $J(\gamma_2)$ since by construction $K$ pulls back to
$K_{12}$ over $Y.$ Now, consider an embedding of $V,$ using the collar
neighborhood theorem, as a neigborhood, $\Omega\subset Y\times\bbT^2,$ of
the image of $Y^{[2]}$ under this embedding. The bundle $K,$ pulled back to
$Y\times\bbT^2$ by the projection onto $\bbT^2$ can be deformed to a bundle
$\tilde K,$ which is equal over $\Omega$ to the pull back under the normal
retraction of its restriction, $K_{12},$ to the image of $Y.$ Then the
embeddign \eqref{mms2002.210} embeds $\CI(Y^{[2]};J(\gamma _2)$ as a
subalgebra of $\CI((Y')^{[2]};J(\gamma_1)\otimes\tilde K\boxtimes\tilde
K'),$ $Y'=Y\times\bbT^2.$ Moreover, using the full spectral expansion of
the harmonic oscillator, the completion of the image is Morita equivalent
to the whole subalgebra with support in the compact manifold with boundary
which is the closure of $\Omega\subset Y'.$ This in turn is Morita
equivalent to the whole algebra and hence, after another deformation of
$\tilde K$ back to $K$ over $\bbT^2$ to $\cA_{12}$ in
\eqref{mms2002.209}. This gives the first isomorphism in
\eqref{mms2002.209}. The second follows from stabilization by the compact
operators on $K$ over $\bbT^2$ as discussed above, completing the proof.
\end{proof}

\begin{proof}[Proof of \eqref{mms2002.203}] As noted above this is a
  corollary of Proposition~\ref{mms2002.205} and the preceeding
  discussion. Namely this provides a stabilized isomorphism, unique up to
  homotopy, of the Azumaya bundle in \S\ref{Circlecase} with that
  constructed over $\tilde L\times_XY$ by lifting the trivialization over
  $\tilde L$ to the fiber product. The same is true by lifting the
  trivialization over $Y$ to the fiber product. Then the Proposition
  constructs a stable isomorphism, again unique up to homotopy, of the
  two lifts to $\tilde L\times_X\times Y\times\bbT^2.$ These stable
  isomorphisms project to a unique isomorphism of the twisted K-groups,
  as in \eqref{mms2002.203}, consistent under composition.
\end{proof}

\begin{lemma}\label{TrivChange} The Azumaya bundle $\smA(\gamma),$ lifted
  to $Y,$ is completion isomorphic to the trivial bundle $\cK,$ with the
  isomorphism fixed up to homotopy, and this induces the natural
  isomorphisms \eqref{mms2002.192}.
\end{lemma}

\begin{proof} The primitivity condition on $J$ shows that when lifted to
  the second two factors of $Y^{[3]}$ it is isomorphic to the bundle over
  $Y^{[3]}$ of which the elements of $\Psi(Y^{[2]}/Y;J'),$ the smoothing
  operators on the fibers of $Y^{[2]}$ as a bundle over $Y,$ are
  (density-valued) sections. As noted above, Kuiper's theorem shows that
  the completion of $\Psi(Y^{[2]}/Y;J')$ is naturally, up to homotopy,
  isomorphic to the trivial Azumaya bundle $\cK,$ from which
  \eqref{mms2002.192} follows.
\end{proof}

\section{Analytic index}\label{sect:analytic ind}

We now proceed to define the analytic index map \eqref{mms2002.55} using
the constructions in \S\ref{decompGeom}, \S\ref{SmoothAzbun} and
\S\ref{Stabiso}. The first step is to define the projective bundle of
pseudodifferential operators. We do this by direct generalization of
Definition~\ref{def:Azumaya}. So, for any $\bbZ_2$-graded bundle
$\bbE=(E_+,E_-)$ over $Y$ set
\begin{equation}
\Psi^{\ell}(Y/X; \cA\otimes\bbE)=
I^{\ell}(Y^{[2]},\Diag;J\otimes\Hom(\bbE)\otimes\Omega_R)
\label{mms2002.78}\end{equation}
where $\Hom(\bbE)=E_-\boxtimes E_+'$ over $Y^{[2]}$ and $I^{\ell}$ is the
space of (classical) conormal distributions. As is typical in projective index theory,
the Schwartz kernel of the projective family of elliptic operators is globally
defined, even though one only has local families of elliptic operators 
with a compatibility condition on triple overlaps given by a phase factor.  
More precisely, 
definition \eqref{mms2002.78} means that on any open set in
$Y^{[2]}$ over which $J$ is trivialized as $\Hom(K_\tau)$ as in 
Proposition \ref{mms2002.66}, the kernel is that of a family of
pseudodifferential operators on the fibres of $Y$ acting from sections of
$E_+$ to sections of $E_-.$ It follows from the standard case that
\eqref{mms2002.69} also extends immediately to show that if
$\tau:U\longrightarrow Y$ is a section over an open set, then
\begin{multline}
\Psi^{\ell}(\phi^{-1}(U)/U;\cA\otimes\bbE) \cong_{\tau}
\Psi^{\ell}(\phi^{-1}(U)/U;K_{\tau}\otimes\bbE)\\
\overset{\sigma _\ell}
\longrightarrow \CI(S^*(\phi^{-1}(U)/U);\hom(\bbE)\otimes N_{\ell}),
\label{mms2002.79}\end{multline}
where we have used the fact that $\hom(K_\tau)$ is canonically trivial.
The principal symbol map here is invariant under conjugation by functions and
hence well-defined independent of the trivialization; $N_{\ell}$ is
the trivial line bundle corresponding to functions of homogeneity $\ell$ on
$T^*(\phi^{-1}(U)/U)$ and $\hom(\bbE)$ is the bundle (over
$S^*(\phi^{-1}(U)/U))$ of homomorphisms from $E_+$ to $E_-.$ Thus the usual
composition properties of pseudodifferential operators extend without any
difficulty as do the symbolic properties. More precisely, 

\begin{lemma}\label{mms2002.80} The spaces of smooth sections of
$\Psi^{\ell}(Y/X; \cA\otimes\bbE)$ form graded modules under composition
  and the principal symbol defined through \eqref{mms2002.79} is
  independent of $\tau$ and gives a multiplicative short exact sequence for
  any $\ell:$ 
\begin{equation}
\Psi^{\ell-1}(Y/X; \cA\otimes\bbE) \longhookrightarrow
\Psi^{\ell}(Y/X; \cA\otimes\bbE)\stackrel{\sigma_{\ell}}{\longrightarrow}
\CI(S^*(Y/X;p^*\hom(\bbE)\otimes N_{\ell})).
\label{mms2002.81}\end{equation}
\end{lemma}

\begin{proof} The theory of conormal distributional sections of a complex
  vector bundle with respect to a submanifold, implicit already in
  H\"ormander's paper \cite{HoFIOI}, shows that these have well-defined
  principal symbols which are homogeneous sections over the conormal bundle
  of the submanifold, in this case the fibre diagonal, of the pull-back of
  the bundle tensored with a density bundle. In this case, as for
  pseudodifferential operators, the density bundles cancel. Moreover the
  bundle $J$ is canonically trivial over the (fiber) diagonal in $Y^{[2]}$
  by the primitivity property of $J.$ The symbol in \eqref{mms2002.81}
  therefore does not involve any twisting -- it takes values in the same
  space as in the untwisted case, and is a well-defined homogeneous section
  of the homomorphism bundle of $E$ (hence section of that bundle tensored
  with the homogeneity bundle $N_l)$ on the fibre cotangent bundle -- which
  is canonically the conormal bundle of the fibre diagonal, as claimed.
\end{proof}

With the trivialization $\kappa$ fixed, the symbol of a projective family 
of elliptic pseudodifferential operators determines an element 
in $\Kc^0(T^*(Y/X))$
We now show that the index of such a projective elliptic family is an
element in twisted K-theory of the base, $\Kt^0(X,\cA).$ More precisely,
let $P\in\Psi^{m}(Y/X;\cA\otimes\bbE)$ be a projective family of elliptic
operators. This means that the symbol is invertible in the usual sense, so
from the standard ellipticity construction (using iteration over $\ell$ in
the sequence \eqref{mms2002.81}) $P$ has a parametrix $Q \in
\Psi^{-m}(Y/X;\cA\otimes\bbE_-),$ where $\bbE_-=(E_-,E_+),$ such that $S_0 =
1- QP \in \Psi^{-\infty}(Y/X;\cA\otimes E_+)$ and $S_1 = 1- PQ \in 
\Psi^{-\infty}(Y/X;\cA\otimes E_-).$ Then the index is realized using the
idempotent
$$
E_1= \left(\begin{array}{ccc} 1-S_0^2 & Q(S_1 + S_1^2) \\
S_1 P  & S_1^2 \end{array}\right)
\in M_2(\Psi^{-\infty}(Y/X;\cA\otimes\bbE)^\dagger).
$$
Here, $\dagger$ denotes the unital extension of the algebra. It is standard
to verify that $E_1$ is an idempotent.

Then, as in the usual case, the analytic index of $P$ expressed in terms of
idempotents is 
\begin{equation}
\begin{gathered}
\ind_a{(P)}= [E_1-E_0]\in\Kt_0( \Psi^{-\infty}(Y/X;\cA\otimes\bbE))\Mwhere\\
E_0 = \left(\begin{array}{cc}1 & 0 \\0 & 0\end{array}\right) \in
 M_2(\Psi^{-\infty}(Y/X;\cA\otimes\bbE)^\dagger).
\end{gathered}
\label{mms2002.183}\end{equation}
That $\ind_a{(P)}$ is a well-defined element in the K-theory follows from
invariance of K-theory under Morita equivalence of algebras. Thus, the inclusion
$$
\CI(X, \cA)=\Psi^{-\infty}(Y/X;\cA)
\longhookrightarrow \Psi^{-\infty}(Y/X;\cA\otimes\bbE),
$$ 
induces a natural isomorphism of $\Kt_0(\Psi^{-\infty}(Y/X;\cA\otimes \bbE))$ and
$\Kt^0(X;\cA).$ Therefore we have defined the {\em analytic index} of any
projective family of elliptic pseudodifferential operators.

To see that this fixes the map,  
\begin{equation}
\ind_a:\Kc^0(T^*(Y/X);\rho ^*\phi^*\cA) \longrightarrow \Kt^0(X,\cA)
\label{mms2002.135}\end{equation}
we need, as usual, to check homotopy invariance, invariance under bundle
isomorphisms and stability. However, this all follows as in the standard case.

Of particular geometric interest are examples arising from projective
families of (twisted) Dirac operators. If the fibres of $Y$ are
even-dimensional and consistently oriented, let $\Cl(Y/X)$ denote the
bundle of Clifford algebras associated to some family of fiber metrics and
let $\bbE$ be a $\bbZ_2$-graded hermitian Clifford module over $Y$ with
unitary Clifford connection $\nabla^\bbE.$

This data determines a family of (twisted) Dirac operators $\eth_\bbE$
acting fibrewise on the sections of $\bbE.$ We can further twist
$\eth_\bbE$ by a connection $\nabla^\tau$ of
the line bundle $K_\tau$ over $\phi^{-1}(U) \subset Y$
for contractible open subsets $U \subset X.$
In this way, we get a projective family of
(twisted) Dirac operators $\eth_{\cA\otimes\bbE} \in \Psi^1(Y/X; \bbE \otimes\cA)$
which can be viewed as a
family of twisted Dirac operators acting on a projective Hilbert bundle
$\bbP(\phi_*(\bbE\otimes K_\tau))$ over $X.$ Here the local bundle
$\phi_*(\bbE\otimes K_\tau)$ is given by $U \times
L^2(\phi^{-1}(U)/U; \bbE\otimes K_\tau)$ for contractible open
subsets $U \subset X.$

The above projective Dirac family can be globally defined as follows.
Consider the delta distributional section
$\delta_Z^{\bbE, J} \in I^\bullet(Y^{[2]}, J\otimes\Hom(\bbE)\otimes\Omega_R)$,
which is supported on the fibrewise diagonal in $Y^{[2]}$. Let ${}^L\nabla^\bbE$
denote the unitary Clifford connection acting on the left variables, and
$\nabla^J$ a connection on $J$ which is compatible with the primitive
property of $J$. Then $$(1\otimes {}^L\nabla^\bbE+ \nabla^J\otimes 1)
\delta_Z^{\bbE, J} \in I^{\bullet-1}(Y^{[2]}, J\otimes\Hom(\bbE)\otimes T^*(Y/X)\otimes \Omega_R),$$
and composition with the contraction given by Clifford multiplication gives
$$c\circ(1\otimes {}^L\nabla^\bbE+ \nabla^J\otimes 1) \delta_Z^{\bbE, J}\in I^{\bullet-1}(Y^{[2]}, J\otimes\Hom(\bbE)\otimes \Omega_R),$$
which represents the Schwartz kernels of the projective
family of (twisted) Dirac operators denoted above by $\eth_{\cA\otimes\bbE}$.

\section{The topological index}\label{sect:top ind}

In this section we define the {\em topological index} map for the setup in
the previous section,
\begin{equation}
\ind_t: \Kc^0(T(Y/X);\rho ^*\phi^*\cA)\longrightarrow\Kt^0(X; \cA).
\end{equation}
It is defined in terms of Gysin maps in twisted $K$-theory, which have been
studied in the case of torsion twists in \cite{mms1}, which extends
routinely to the general case as in \cite{BEM, CW}. In the particular
case that we consider here, there are several simplifications that we shall
highlight.

We first recall some functorial properties of twisted $K$-theory. Let
$F:Z\longrightarrow X$ be a smooth map between compact manifolds.
Then the pullback map,
$$
F^!:\Kt^0(X, \cA) \longrightarrow \Kt^0(Z, F^*\cA ),
$$
is well defined. 

\begin{lemma}\label{Bott} There is a canonical isomorphism,
$$
j_! : \Kt^0(X, \cA )\cong \Kc^0 (X\times \R^{2N}, \pi_1^*\cA),
$$
determined by Bott periodicity, where the inclusion $j : X \to X
\times \R^{2N}$ is onto the zero section. Here $\pi_1: X\times \R^{2N} \to
X$ is the projection onto the first factor.
\end{lemma}

\begin{proof} First notice that $\Kt^\bullet(X, \cA ) =
  \Kt_\bullet(C^\infty(X, \cA))$ and  $\Kt_c^\bullet (X\times \R^{2N},
  \pi_1^*\cA) =  \Kt_\bullet (C_c^\infty(X \times \R^{2N}; \pi_1^*\cA) ).$
  But $C_{\text{c}}^\infty(X \times \R^{2N}; \pi_1^*\cA)=  C^\infty(X, \cA)\otimes
  C_c^{\infty}(\R^{2N})$. So the lemma follows from Bott periodicity for
  the $K$-theory of (smooth) operators algebras. 
\end{proof}

For the  fiber bundle $\phi:Y\longrightarrow X$ of compact manifolds, 
we know that there is an embedding $i:Y\longrightarrow X \times \R^{N}$,
cf\@.~\cite{Atiyah-Singer1} \S 3. The {fibrewise} differential is an
embedding $Di:T(Y/X)\longrightarrow X\times\R^{2N}$ with complex normal
bundle $\cN$.

Let $\cA$ be the smooth Azumaya algebra over $X$ as defined earlier in 
\S\ref{SmoothAzbun}; there is a fixed trivialization of $\phi^*\cA.$ 
Let $ \cA_\cN$ be the lift of $\phi^*\cA$ to $\cN$. Let
$\rho:T(Y/X)\longrightarrow Y$ be the projection map. Then since
$\phi^*\cA$ is trivialized, we have the commutative diagram 
\begin{equation}
\begin{CD} \Kt_c^0(T(Y/X), \rho^*\phi^* \cA)   @>Di_!:>>  \Kt_c^0 (\cN,  \cA_\cN)   \\         
      @V{\cong_{\kappa^{-1}}}VV          @VV{\cong} V     \\
 \Kt_c^0(T(Y/X))    @>Di_!:>>  \Kt_c^0 (\cN),    
\end{CD}
\label{Extra}\end{equation}
where $Di_!$ in the lower horizontal arrow 
is given by $\xi = (\xi^+, \xi^-, G) \longmapsto \pi^*\xi \otimes (\pi^* {\cS^+},
\pi^*{\cS^-}, c(v))$. 
Here $\xi=(\xi^+,\xi^-)$ is pair of vector bundles over $T(Y/X),$
$G:\xi^+\longrightarrow \xi^-$ a bundle map between them which is an
isomorphism outside a compact subset and $(\pi^* {\cS^+}, \pi^*{\cS^-}, c(v))$
is the usual Thom class of the complex vector bundle $\cN,$ where 
$\pi$ is the projection map of $\cN$ and $\cS^\pm$ denotes the bundle
of half spinors on $\cN$. On the the
right hand side the the graded pair of vector bundle data is  
\begin{equation*}
(\pi^*\xi^+\otimes\pi^*\cS^+\oplus\pi^*\xi^-\otimes\pi^*\cS^-,
\pi^*\xi^+\otimes\pi^*\cS^-\oplus\pi^*\xi^-\otimes\pi^*\cS^+)
\label{Extra2}\end{equation*}
with map between them being
\begin{equation*}
\left[\begin{matrix}
G&c(v)\\ c(v)&G
\end{matrix}\right],\ v\in \cN.
\end{equation*}
This is an isomorphism outside a compact subset of $\cN$ and defines
a class in $\Kt_{\text{c}}^0 (\cN)$
which is independent of choices, provided the trivialization of $\phi^*(\cA)$ is
kept fixed. 
Then the usual {Thom isomorphism} theorem
asserts that $\;Di_!\;$ is an isomorphism. The upper horizontal arrow 
is defined in the same way by tensoring with the same Thom class.

Now, $\cN$ is diffeomorphic to a tubular neighborhood $\cU$ of the image of
$Y;$ let $\Phi: \cU \longrightarrow \cN$ denote this diffeomorphism. 
Then the induced map in $\Kt$-theory gives isomorphisms, 
$$
\Phi^* : \Kt_c ^0 (\cN) \cong \Kt_c^0  (\cU), \qquad \Phi^*:\Kt_c^0  (\cN,  \cA_\cN)\cong
\Kt_c^0  (\cU,  \Phi^*( \cA _\cN)). 
$$

We will next show that the inclusion $i_\cU: \cU \to X \times R^{2N}$ 
of the open set $\cU$ in $X \times R^{2N}$ induces a natural extension map
$$
(i_\cU)_! : \Kc^0 (\cU,  \Phi^*( \cA_\cN))\longrightarrow\Kc^0(X\times R^{2N},
\pi_{1}^{*} \cA)
$$ 
To see this, we need to show that the restriction $i_\cU^*\pi_1^* \cA$ is
trivialized.  Note that $\phi \circ \tau = \pi_1 \circ i_\cU$, where
$\tau : \cU\longrightarrow Y$ is equal to the composition, $\lambda
\circ\Phi $, and $\lambda: \cN\longrightarrow Y$ the projection map. Since
$(\phi \circ \tau)^*\cA = \tau^*\phi^*\cA$ is trivializable because
$\phi^*\cA$ is trivialized, it follows that $i_\cU^*\pi_1^* \cA$ is trivialized.

We have the
following commutative diagram.

\begin{equation} \label{ind_t}
\xymatrix @=3pc { 
\Kt_c^0(T(Y/X)) \ar@{->}[r]^{Di_!}_{\cong} \ar@{->}[d]^\cong_{\rho^*\kappa_*}& \Kt_c^0  (\cN) \ar@{->}[d]^{\cong}  \ar@{->}[r]^{\Phi^*}_{\cong}   &
 \Kt_c^0  (\cU)  \ar@{->}[d]^{\cong} 
 \\
\Kt_c^0(T(Y/X), \rho^*\phi^* \cA) \ar@{->}[ddrr]_{\widetilde\ind_t}
\ar@{->}[r]^{Di_!}_{\cong}  &  \Kt_c^0 (\cN,  \cA_\cN) \ar@{->}[r]^{\Phi^*}_{\cong}  & \Kt_c^0 (\cU,  \Phi^*( \cA_\cN))
 \ar[d]^{(i_\cU)_!} \\& & \Kc^0(X\times R^{2N}, \pi_{1}^{*} \cA)
  \ar[d]^{j_!^{-1}}_{\cong} \\& &  \Kc^0(X, \cA).
}
\end{equation}

%


The composition of the maps in the diagram above defines the Gysin map in twisted K-theory,
$$
Di_! :\Kc^0(T(Y/X)) \longrightarrow\Kc^0(X\times\R^{2N},
\pi_1^* \cA).
$$
Here we have used the fact that since $\pi = \pi_1 \circ i$ it follows that
$Di^*\pi_1^* \cA = \rho^*\phi^* \cA $ is trivialized. Now define the
\emph{topological index}, as the map
\begin{equation}\label{defn:indt}
\ind_t=j_!^{-1}\circ Di_! : \Kc^0(T(Y/X))
\longrightarrow \Kt^0(X; \cA),
\end{equation}
where we apply the Thom isomorphism in Lemma~\ref{Bott} to see that the
inverse $j_!^{-1}$ exists. We also note that $\widetilde\ind_t \circ
\rho^*\kappa_* = \ind_t,$ consistent with the corresponding analytic
indices.

The source is untwisted since $\cA$ is trivialized by $\kappa,$ as an
Azumaya bundle, when pulled back to $Y.$ The identification of twisted and
untwisted K-theory in \eqref{mms2002.55} depends on the choice of
trivialization \eqref{mms2002.87} but then so does the Azumaya bundle and
these choices do not change the index map $\widetilde\ind_{\text{t}}.$

\section{Twisted Chern character}\label{TwistedChern}

First we recall an explicit formula for the odd Chern character in the
untwisted case. For any compact manifold (of positive dimension), $Z,$ the
group of invertible, smoothing, perturbations of the identity operator
\begin{equation}
G^{-\infty}(Z)=\{a\in\Psi^{-\infty}(Z);\exists\ (\Id+a)^{-1}=\Id+b,\
b\in\Psi^{-\infty}(Z)\} 
\label{mms2002.151}\end{equation}
is classifying for odd K-theory. So there is a canonical identification of the
odd K-theory of a compact manifold $X$ with the (smooth) homotopy classes
of (smooth) maps 
\begin{equation}
\Ko(X)=[X;G^{-\infty}(Z)].
\label{mms2002.152}\end{equation}

The odd Chern character is then represented in deRham cohomology by the
pull-back of the universal Chern character on $G^{-\infty}(Z):$ 
\begin{equation}
\begin{aligned}
\Ch=\sum\limits_{k=0}^\infty &c_k\tr\left((A^{-1}dA)^{2k+1}\right)\\
&=-\frac1{2\pi i}
\int_0^1\tr\left((A^{-1}dA)\exp(\frac{t(1-t)}{2\pi i}(A^{-1}dA)^2)\right)dt,\
A=\Id+a.
\end{aligned}
\label{mms2002.153}\end{equation}
Here $dA=da,$ as for finite dimensional Lie groups, is the natural
identification of $T_aG^{-\infty}$ with $\Psi^{-\infty}(Z)$ coming from the
fact that $G^{-\infty}(Z)$ is an open (and dense) set in
$\Psi^{-\infty}(Z).$ Thus for an odd K-class
\begin{equation}
\begin{aligned}
a:X\longrightarrow &G^{-\infty}(Z),\ \Ch([a])=[a^*\Ch]\in
\operatorname{H}^{\odd}(X),\\
&a^*\Ch=-\frac1{2\pi i}
\int_0^1\tr\left((\Id+a)^{-1}da
\exp(\frac{t(1-t)}{2\pi i}((\Id+a)^{-1}da)^2\right)dt
\end{aligned}
\label{mms2002.154}\end{equation}
where now the differential can be interpreted in the usual way for
functions valued in the fixed vector space $\Psi^{-\infty}(Z).$

For any fiber bundle $\phi:Y\longrightarrow X,$ with typical fiber $Z,$ 
$\Ko(X)$ is also naturally identified with the abelian group of homotopy classes
of sections of the bundle of groups over $X$ with fiber $G^{-\infty}(Z_x)$
at $x\in X.$ That is, the twisting by the diffeomorphism group does not
affect this property. The formula \eqref{mms2002.154} can be extended to
this geometric setting by choosing a connection on $\phi,$ i\@.e\@.~a lift of
vector fields from $X$ to $Y.$ Indeed, such a connection can be identified
as a connection on the bundle $\CI(Y/X),$ with fibres $\CI(Z_x)$ (and space
of sections $\CI(Y)),$ as a differential operator
\begin{equation}
\begin{aligned}
\nabla:\CI(Y)\longrightarrow
&\CI(Y;\phi^*T^*X),\\
&\nabla(hg)=(dh)g+h\nabla g,\ h\in\CI(X),\ g\in\CI(Y).
\end{aligned}
\label{mms2002.155}\end{equation}
The curvature of such a connection (extended to a superconnection),
is a first order differential operator on the fibres
$w=\nabla^2/2\pi i\in\Diff1(Y/X;\bbC,\phi^*\Lambda ^2X)$ from the trivial
bundle to the 2-form bundle lifted from the base. The connection on $Y$
induces a connection on $\Psi^{-\infty}(Y/X),$ as a bundle of operators on
$\CI(Y/X),$ acting by conjugation and then
\eqref{mms2002.154} is replaced by
\begin{multline}
\Ch(A)=\\
-\frac1{2\pi i}
\int_0^1\tr\bigg((A^{-1}\nabla A)
\exp\big((1-t)w+tA^{-1}wA+\frac{t(1-t)}{2\pi i}(A^{-1}\nabla A)^2\big)\bigg)dt,
\\
A:X\longmapsto G^{-\infty}(Y/X),\ \pi\circ A=\Id.
\label{mms2002.156}\end{multline}
Note that any such section is homotopic to a section which is a finite rank
perturbation of the identity, in which case \eqref{mms2002.156} becomes the
more familiar formula. The same conclusions, and formula hold, if the
bundle of groups of smoothing perturbations of the identity acting on a
vector bundle over $Y,$ $G^{-\infty}(Y/X;E),$ is considered, provided the
connection (and curvature) are lifted to a connection on $E.$

Note that \eqref{mms2002.156} can also be considered as the pull-back of a
universal form on the total space of the fibration $G^{-\infty}(Y/X).$ It
then has the property that restricted to a fiber, so that the curvature
vanishes, one recovers the original form in \eqref{mms2002.153}.

The case of immediate interest arises from a circle bundle $p:\tilde
L\longrightarrow X.$ As explained in \S\ref{SmoothAzbun} we consider the
fiber product $\tilde L^{[2]}$ fibred over $\tilde L$ with the fibres taken
to be in the second factor, with the smoothing operators acting on sections
of $J.$ Of course these operators are acting on the restriction of this
line bundle to each fiber, which is a circle, so they can always be
identified on each fiber with ordinary smoothing operators. On the other
hand $J$ has the primitivity property of Lemma~\ref{mms2002.71} which
allows us to identify the smoothing operators on sections of $J$ on the
fibres of $\tilde L^{[2]}$ with $\CI(\tilde L^{[3]};\pi_F^*J)$ as in
Proposition~\ref{mms2002.76}. An explicit fiber density factor is not needed
since this is supplied naturally by the Hermitian structure.

\begin{proposition}\label{mms2002.163} Suppose $a\in\CI(\tilde L^{[2]};J)$ is
  such that $A=\Id+a$ is everywhere invertible over $X.$ Then the odd Chern
  character of $\Id+a,$ as a form on $\tilde L$ computed with respect to a
  unitary connection on $L$ and the primitive connection of
  Lemma~\ref{mms2002.64} on $J,$ with combined curvature $\Omega,$
\begin{multline}
\Ch_{\cA}(A)=\\
-\frac{1}{2\pi i}\int_0^1\tr\left((A^{-1}\nabla A)
\exp\left((1-t)\Omega +tA^{-1}\Omega A+t(1-t)(A^{-1}\nabla A)^2\right)\right)dt\\
\in\CI(\tilde L;\Lambda^{\odd})
\label{mms2002.164}\end{multline}
is closed and satisfies the conditions in \eqref{mms2002.147}.
\end{proposition}

\begin{proof} That the Chern form \eqref{mms2002.164} is closed follows
  from the standard properties. To see the other stated properties, we
  choose a section of $\tilde L$ over an open set $U\subset X$ over which
  $u$ has a smooth logarithm and set $f=\frac1{2\pi i}\log u.$ In terms of
  the induced trivializations
\begin{equation}
\tilde L_U=U\times\bbS,\ \tilde L^{[2]}_U=U\times\bbS\times\bbS,
\label{mms2002.168}\end{equation}
let the fiber variables be $\theta _1$ and $\theta _2.$ The operators are
acting in the $\theta_2$ variable and the lifted connection on $\tilde L$
as a fibration over $\tilde L$ is therefore 
\begin{equation}
\nabla=d_x+d_{\theta _1}+\overline{\gamma}\pa_{\theta_2}
\label{mms2002.178}\end{equation}
where $\overline{\gamma}\in\CI(U;\Lambda)$ is the local connection form for
$L.$ The corresponding connection on $\CI(\tilde L^{[2]}/\tilde L;J)$ in
terms of this trivialization and of the connection on $J$ from
Lemma~\ref{mms2002.64}
\begin{equation}
\nabla_J=d_x+d_{\theta _1}+\overline{\gamma}\pa_{\theta_2}+fd\theta
_1-\overline{\gamma}f.
\label{mms2002.179}\end{equation}
The curvature is
\begin{equation}
\Omega=\nabla_J^2/2\pi
i=\overline{\beta}(\pa_{\theta_2}-f)+
\frac1{2\pi i}\overline{\alpha}\wedge\gamma,\ \gamma=d\theta_1+\overline{\gamma}.
\label{mms2002.180}\end{equation}
In terms of this local trivialization $a=a(x,\theta_2,\theta_3)$ is
independent of the first (parameter) fiber variable. Inserting
\eqref{mms2002.180} into \eqref{mms2002.164} observe that the two terms
in \eqref{mms2002.180} commute so 
\begin{equation}
\Ch_{\cA}(A)=e^{\frac{\overline{\alpha}\wedge\gamma}{2\pi i}}v,\
v\in\CI(U;\Lambda ^{\odd})
\label{mms2002.181}\end{equation}
satisfies the conditions of \eqref{mms2002.147}.
\end{proof}

Note that under a deck transformation of $\hat X,$ i\@.e\@.~integral shift
of $f$ by $n\in\bbN,$ each term undergoes conjugation by $\exp(in\theta)$
and the Chern form itself is therefore unchanged. 

It follows from Proposition~\ref{mms2002.146} that $\Ch_{\cA}(A)$ defines
an element in the twisted cohomology of $X,$ given explicitly by the form $v$
in \eqref{mms2002.181}. Although the proof above is written out for
sections of $J$ over $\tilde L^{[2]}$ the passage to matrix-valued sections
is merely notational, so it applies essentially unchanged to elements of 
\begin{multline}
G(X;\cA\otimes M(N,\bbC))=\\
\{a\in\CI(\tilde L^{[2]};J\otimes M(N,\bbC);
\Id_{N\times N}+a(x)\text{ is invertible for all }x\in X\}.
\label{mms2002.182}\end{multline}

\begin{lemma}\label{mms2002.159} The Chern form \eqref{mms2002.164}
  descends to represent the twisted Chern character
\begin{equation}
G(X;\cA\otimes M(N,\bbC))/\kern-3pt\sim\ =\Ko(X;\cA)\longrightarrow\Ho(X;\delta )
\label{mms2002.157}\end{equation} 
where the equivalence relation on invertible matrix-valued sections of the
Azumaya bundle is homotopy and stability.
\end{lemma}

\begin{proof} The invariance of the twisted cohomology class under
  stabilization follows directly from the definition. Invariance under
  homotopy follows as usual from the fact that the construction is
  universal and the form is closed, so is closed for a homotopy when
  interpreted as a family over $X\times[0,1]$ and this proves the
  invariance of the cohomology class.

It also follows directly from the definition that the twisted Chern
character behaves appropriately under the Thom isomorphism for a complex
(or symplectic) vector bundle $w:W\longrightarrow X.$ That is, there is a
commutative diagram with horizontal isomorphisms
\begin{equation}\xymatrix{
\Ko(X;\cA)\ar[r]^-{\boxtimes b}\ar[d]_{\Ch_{\cA}}
&\Ko(W;w^*\cA)\ar[d]^{\Ch_{\cA}}\\
\Ho(X;\delta)\ar[r]^-{\wedge\Td}
&\Hoc(W;w^*\delta)
}
\label{mms2002.176}\end{equation}

As in the case of the bundle of groups $G^{-\infty}(Y/X)$ the form
\eqref{mms2002.164} is again the pull-back from the total space of the
bundle of groups $G^{-\infty}(X;\cA)$ of invertible sections of the unital
extension of the Azumaya bundle and then restricting this universal form to
a fiber one again recovers the standard odd Chern character in
\eqref{mms2002.153}. This is enough to show that the Chern form here does
represent the twisted Chern character as widely discussed in the
literature, for instance recently by Atiyah and Segal in
\cite{AtSe}. Namely they remark that the Chern character as they describe
it (in the even case), which is determined by universality under pull-back
from the twisted $\PU$ bundle over $K(\bbZ,3),$ is actually determined by
its pull-back to the 3-sphere. The $\PU$ bundle over $\bbS^3$ with
generating DD class is trivial over points and so can be transferred to
$(0,\pi)\times\bbS^2$ to be trivial outside a compact set and thence to
$\bbS\times\bbS^2$ where it reduces to the twisted bundle again with
generating DD class. As shown in \cite{AtSe}, the universal twisted Chern
character over the $3$-sphere is determined by multiplicativity and the
fact that it restricts to the standard Chern character on the fibres over
points. The odd case follows by suspension so the deRham version of the
Chern character above does correspond with more topological definitions.
\end{proof}

We need some extensions of this discussion of the odd twisted Chern
character. In particular we need to discuss the even case. However, the
context needs to be broadened to cover operators on the fibres of a
trivializing bundle $Y$ as in \eqref{mms2002.115} --
\eqref{mms2002.87}. Finally the relative case is needed for the discussion
of the Chern character of the symbol and the index formula in twisted
cohomology. Fortunately these are all straightforward generalizations of
the untwisted case.

We start with the extension of the odd twisted Chern character to the more
general geometric case under discussion here. Thus, instead of being over
$\tilde L^{[2]},$ the bundle $J$ is defined over $Y^{[2]}.$ Still, when
lifted to the fiber product
\begin{equation}
\tilde Y^{[2]}=\tilde L\times_XY^{[2]},
\label{mms2002.186}\end{equation}
$J$ is reduced to a the exterior tensor product 
\begin{equation}
p^*J=\tilde J\boxtimes\tilde J'\text{ over }\tilde Y^{[2]}
\label{mms2002.188}\end{equation}
where $\tilde J$ is a line bundle over $\tilde Y=\tilde L\times_X Y.$ Namely,
there is a character property for $s:Y^{[2]}\longrightarrow \UU(1),$ which
is determined by the trivialization of $L$ over $Y,$ when lifted to $\tilde
Y^{[2]}:$ 
\begin{equation}
s(z_1,z_2)=\tilde s(z_1,\tilde l)\tilde s(z_2,\tilde l)^{-1},\ \tilde s:
\tilde Y\longrightarrow \UU(1).
\label{mms2002.187}\end{equation}
Here $\tilde s$ is fixed by the demand that it intertwine the
trivializations over $\tilde L$ and $Y.$ Thus, using $\tilde s$ to define
$\tilde J$ by the same procedure as previously used to define $J,$
\eqref{mms2002.188} follows.

From this point the discussion proceeds as before. That is, the Azumaya
bundle $\cA_Y$ acting on the fibres of $Y$ over $X$ lifts to $\tilde Y,$
acting on the same fibres but now over $\tilde L,$ into a subalgebra of
$\Psi^{-\infty}(\tilde Y/\tilde L;\tilde J).$ Then, as above, the odd Chern
character for invertible sections of the unital extension of the Azumaya
bundle is a differential form on $\tilde L$ satisfying \eqref{mms2002.147}
and so defines the twisted odd Chern character in this more general
geometric setting.

Next, consider the even twisted Chern character. To do so, recall that for
a complex vector bundle $E$ embedded as a subbundle of some trivial
$\bbC^N$ over a manifold $X$ the curvature, and Chern character, can be
written in terms of an idempotent $e$ projecting onto the range as the
2-form valued homomorphism $\omega_E=e(de)(1-e)(de)e/2\pi i.$ There is a
similar formula if $E$ is embedded in a possibly non-trivial bundle $F$
with connection $\nabla_F$ which is projected onto $E$ using $e.$ The same
formula applies in the case of a subbundle of $\CI(Y/X;E)$ given by a
family of idempotents $e.$ In the untwisted case, the K-theory of $X$ can
be represented by formal differences of finite rank idempotents in the
fibres of $\CI(Y/X;E),$ giving finite dimensional bundles. In general, in
the twisted case, the K-theory is interpreted as the $C^*$ K-theory of a
non-unital algebra (the completion of $\cA$ in the compacts), it is
necessary to take pairs of infinite rank idempotents in
$\bbC^N\otimes\cA^\dagger$ with differences valued in $\bbC^N\otimes\cA.$
In fact it is enough to take single idempotents in
$\bbC^N\otimes\cA^\dagger$ with constant unital part $e_0\in M(N,\bbC)$ and
consider the formal difference $e\ominus e_0$ to generate the K-theory. For
the untwisted case, the usual Chern character is given by 
\begin{equation}
\tr(\exp(\frac{\omega_E^2}{2\pi i})-e_0)
\label{mms2002.189}\end{equation}
as can be shown by suspension from the odd case if desired. Here all terms
in $\Lambda ^{>0}$ involve a derivative of $e$ and hence are smoothing, as
is the normalized term of form degree zero, so the trace functional can be
applied.

To carry this discussion of the even Chern character to the twisted case,
we can proceed precisely as above. Namely, given an idempotent section,
$e,$ of $\bbC^N\otimes\cA_Y^\dagger$ as a bundle over $X$ with constant
unital term $e_0$ one can compute the Chern form \eqref{mms2002.189} after
lifting the idempotent to $\bbC^N\otimes\Psi^{-\infty}(\tilde Y/\tilde
L;\tilde J\otimes\bbC^N)$ as discussed above. Then, for the same reason,
the form satisfies \eqref{mms2002.147} and defines the even Chern character
as a twisted deRham form on $X.$

The final extension is to the relative case to handle the Chern character
of the symbol of a pseudodifferential operator. As discussed in
\cite{Albin-Melrose} for any real vector bundle $W\longrightarrow Y$ (here
applied to $T^*(Y/X))$ the compactly supported cohomology of $W$ can be
obtained directly as from the relative deRham complex of $SW,$ the sphere
bundle of $W,$ and $Y.$ This involves the same odd Chern class on $SW$ (which is no
longer closed) and the even Chern class on $Y$ which `corrects' the failure
of the odd form to be closed. The extension to the twisted case just
combines the two cases discussed above; this is briefly considered in
\S\ref{sect:coh}.

\section{Semiclassical quantization}\label{scl}

To avoid the usual complications which arise in the proof of the index
theorem, especially concerning the multiplicativity of the analytic index
-- although they are no worse in the present twisted setting than the
standard one -- we introduce another definition of the index map using
semiclassical pseudodifferential operators. This approach is discussed in
more detail in \cite{fmtga} but the underlying notion of a semiclassical
family of pseudodifferential operators is well established in the
literature \cite{GrigSjos}. The method of `asymptotic morphism' of Connes
and Higson is closely related to the notion of semiclassical limit.

\begin{proposition}\label{mms2002.93} Let $\psi:M\longrightarrow B$ be a
  fibre bundle of possibly non-compact manifolds then the modules
$$
\Psi^\ell_{\comp,\scl}(M/B;\bbE)\subset
  \CI((0,1)_{\epsilon};\Psi^\ell_{\comp}(M/B;\bbE))
$$
of semiclassical families of classical, uniformly compactly-supported,
pseudodifferential operators on the fibres of $\psi$ are well defined for
any $\bbZ_2$-graded bundle $\bbE,$ have a global multiplicative exact
symbol sequence
\begin{equation}
0\longrightarrow
\epsilon \Psi^\ell_{\comp,\scl}(M/B;\bbE)\hookrightarrow 
\Psi^\ell_{\comp,\scl}(M/B;\bbE)\overset{\sigma _{\scl}}\longrightarrow  
S^{\ell}_{\comp}(T^*(M/B);\hom(\bbE))\longrightarrow 0
\label{mms2002.94}\end{equation}
and completeness property 
\begin{equation}
\bigcap_j\epsilon
^j\Psi^\ell_{\comp,\scl}(M/B;\bbE)=\dCI([0,1);\Psi^\ell_{\comp}(M/B;\bbE)).
\label{mms2002.96}\end{equation}
\end{proposition}

\noindent Note that the space of functions on the right in
\eqref{mms2002.94} consists of the \emph{global} classical symbols on
$T^*(M/B),$ with compact support in the base $M,$ not the quotient by the
symbols of order $\ell-1.$ The space on the right in \eqref{mms2002.96}
consists of the smooth families of pseudodifferential operators with
uniformly compact support in the usual sense, depending smoothly on the
additional parameter $\epsilon \in[0,1)$ down to $\epsilon =0$ where they
vanish with all derivatives. Thus, by iteration, the semiclassical symbol
in \eqref{mms2002.94} captures the complete behaviour of these operators
as $\epsilon \downarrow0.$

To define the semiclassical index maps, one for each parity, we only need
the smoothing operators of this type, for $\ell=-\infty;$ indeed this is
the key to their utility. In this special case the Schwartz kernels of the
operators are easily described explicitly. Namely they correspond to the
subspace of $\CI((0,1)\times M^{[2]}_{\psi};\Hom(\bbE)\otimes\Omega _R))$
consisting of those functions which have support in some set $(0,1)\times
K$ with $K\subset M^{[2]}_{\psi}$ compact, which tend to $0$ rapidly with
all derivatives as $\epsilon \downarrow0$ in any closed set in
$M^{[2]}_{\psi}$ disjoint from the diagonal and which near each point of
the diagonal take the from
\begin{equation}
\epsilon ^{-d}K(\epsilon,b,z,z',\frac{z-z'}\epsilon)|dz'|
\label{mms2002.97}\end{equation}
where $K$ is a smooth bundle homomorphism which is uniformly Schwartz in
the last variable and $d$ is the fiber dimension.

As with usual pseudodifferential operators, there is no obstruction to
defining $\Psi^{\ell}_{\scl}(Y/X;\cA\otimes\bbE)$ either by transferring
the kernels directly to sections of $J\otimes\Hom(\bbE)$ over $Y^{[2]}$ or
by using the local form \eqref{mms2002.69}.

\begin{proposition}\label{mms2002.98} The space of invertible elements in
  the unital extension of the semiclassical twisted smoothing operators
  defines an odd index map via the diagram
\begin{equation}
\xymatrix@C=-1pc{
&\bigcup_N\left\{(A,B)\in\Psi^{-\infty}_{\scl}(M/B;\cA\otimes\bbC^N);
(\Id+A)^{-1}=\Id+B\right\}
\ar[dl]_{[\Id+\sigma_{\scl}(A)]\quad}\ar[dr]^{[(\Id+A)\big|_{\epsilon=\ha}]}
\\
\Kt^1_{\comp}(T^*(Y/X))\ar[rr]^{\ind^1_{\scl}}&&\Kt^1_{\comp}(X;\cA).
}
\label{mms2002.99}\end{equation}
\end{proposition}

\begin{proof} The space on the top line in \eqref{mms2002.99} consists of
  the invertible perturbations of the identity by semiclassical smoothing
  operators, with the inverse of the same form. Thus it follows that
  $\Id+\sigma _{\scl}(A)$ is invertible as a smooth family of $N\times N$
  matrices over $T^*(Y/X),$ reducing to the identity at infinity. It
  therefore defines an element of odd K-theory giving the map on the
  left. Conversely, the invertibility of $\Id+a$ for a symbol $a$ implies,
  using the exactness of the symbol sequence, that $\Id+A,$ where
  $\sigma_{\scl}(A)=a,$ is invertible at least for small $\epsilon.$
  Modifying the semiclassical family to remain invertible for $\epsilon
  \in(0,1)$ shows that this map is surjective. The map on the right,
  defined by restriction to $\epsilon =\ha$ (or any other positive value)
  immediately gives an element of the odd twisted K-theory of the base.
 
To see that the `odd semiclassical index', or push-forward map, is defined
from this diagram it suffices to note that the `quantized' class on the
right only depends on the class on the left up to homotopy and stability,
which as usual follows directly from the properties of the algebra.
\end{proof}

For this odd index there is a companion even index map. Recall that a
compactly supported K-class can be defined by a smooth map into $N\times N$
matrices which takes values in the idempotents and is constant outside a
compact set, where the class can be identified with the difference of the
projection and the limiting constant projection.

\begin{proposition}\label{mms2002.101} If $a\in\CIc(T^*(Y/X);\bbC^N)$ is
  such that $\Pi_\infty+a$ takes values in the idempotents, where $\Pi_\infty\in
  M(N,\bbC),$ then $a$ has a semiclassical quantization 
\begin{equation}
A\in\Psi^{-\infty}_{\scl}(Y/X;\cA\otimes\bbC^N),\ \sigma _{\scl}(A)=a,
\label{mms2002.102}\end{equation}
such that $(\Pi_{\infty}+A)^2=\Pi_{\infty}+A$ and this leads to a well-defined even
semiclassical index map 
\begin{equation}
\xymatrix@1{
**[l]\Kt^0_{\comp}(T^*(Y/X))\ar@<1ex>[r]^-{\ind^0_{\scl}}
\ar@<-1ex>[r]^-=_-{\ind^0_{\text{a}}}&\Kt^0_{\comp}(X;\cA)
}
\label{mms2002.103}\end{equation}
analogous to \eqref{mms2002.99} and as indicated, equal to the analytic
index as defined in \S\ref{sect:analytic ind}.
\end{proposition}

\begin{proof} Certainly a quantization of $a$ exists by the surjectivity of
  the symbol map. Moreover the idempotent $\Pi_{\infty}+a$ can be extended to a
  `formal' idempotent, meaning that, using the symbol calculus, the quantization
  can be arranged to be idempotent up to infinite order error at $\epsilon
  =0.$ The error terms of order $-\infty$ in the semiclassical smoothing
  algebra are simply smoothing operators vanishing to infinite order with
  $\epsilon.$ Use of the functional calculus then allows one to perturb the
  quantization by such a term to give a true idempotent for small $\epsilon
  >0.$ Then stretching the parameter arranges this for $\epsilon \in(0,1).$
  The pair of this projection, for any $\epsilon >0,$ and the limiting
  constant projection, $\Pi_{\infty}$ defines a K-class. The existence of
  the map \eqref{mms2002.103} then follows in view of the homotopy
  invariance and stability of this construction.

To see the equality with the analytic index as previously defined is the
major step in the proof of the index theorem. This amounts to a construction
giving both this semiclassical index map and the usual analytic index map
at the same time. The two index maps, semiclassical and analytic are based
on two different models for the compactly supported K-theory of $T^*(Y/X)$
-- or more generally of a vector bundle $W.$ The first reduces to the set
of projection-valued smooth maps $W\longrightarrow M(N,\bbC)$ into
matrices which are constant outside a compact set. The second is defined in
terms of triples, consisting of a pair of vector bundles over the base
together with an isomorphism between their lifts to $S^*W.$

These two models can be combined into a larger one, in which the set of
objects are triples $(E,F,a)$ where $E$ and $F$ are vector bundles over
$\com{W},$ the radial compactification of $W,$ given directly as smooth
projection-valued matrices into some $\bbC^N$ and where $a$ intertwines
these two smooth families of projections over $S^*W,$ the boundary of the
radial compactification. The equivalence relation
$(E_1,F_1,a_1)\simeq(E_2,F_2,a_2)$ is generated by isomorphisms, meaning
smooth intertwinings of $E_1,$ and $E_2$ and of $F_1$ and
$F_2$ over $\com{W}$ which also intertwine the isomorphisms over $S^*W,$ the
boundary, plus stability. This again gives $\Kt_{\comp}(W).$

Standard arguments show that any such class in this general sense is
equivalent to an `analytic class' in which the bundles are lifted from the
base, or a `semiclassical class' in which the projections are constant
outside a compact set and the isomorphism between them is the identity --
in fact the second projection can be taken to be globally
constant. Moreover equivalence is preserved under these reductions.

Using these more general triples a combined analytic-semiclassical
quantization procedure may be defined by first taking semiclassical
quantizations of the projections $E,$ $F$ to actual semiclassical families
$P,$ $Q$ which are projections; the classical symbols of these projections
can be chosen to be independent of $\epsilon .$ This is again the standard
argument for quantizations of idempotents which is outlined above. Then the
isomorphism $a$ can be quantized to a pseudodifferential operator $A$ in the
ordinary sense but this can be chosen to satisfy $AP(\ha)=A=Q(\ha)A$ so it
`acts between' the images of $P(\ha)$ and $Q(\ha).$ This is accomplished by
choosing some $A'$ with symbol $a$ and replacing it by the
`Toeplitz operator' $A=Q(\ha)A'P(\ha)$ which necessarily has the same symbol.

Then $A$ is relatively elliptic, in the sense that it has a parametrix $B$
satisfying $BQ(\ha)=B=P(\ha)B$ and with $AB-Q(\ha)$ and $BA-P(\ha)$
smoothing operators. The analytic-semiclassical index can now be defined
using the using the same formula as the analytic index above. That it is
well-defined involves the standard homotopy and stability arguments. 

Finally then this map clearly reduces to the analytic and semiclassical
index maps on the corresponding subsets of data and hence these two maps
must be equal. The introduction of the Dixmier-Douady twisting makes
essentially no difference to these constructions so the equality in
\eqref{mms2002.103} follows.
\end{proof}

\begin{proposition}\label{mms2002.104} The appropriate form of Bott
periodicity can be proved directly giving commutative diagrams
\begin{equation}
\xymatrix{
\Kt_{\comp}^1(T^*(Y/X))\ar[r]\ar[d]_{\ind_{\scl}^1}
&\Kt_{\comp}^0(T^*(Y\times\bbR)/(X\times\bbR))
\ar[d]^{\ind_{\scl}^0}\\
\Kt^1(X;\cA)\ar[r]&\Kt^0(X\times \bbR;\cA)
}
\label{mms2002.105}\end{equation}
where the horizontal maps are the clutching construction and
\begin{equation}
\xymatrix{
\Kt_{\comp}^0(T^*(Y/X))\ar[r]\ar[d]_{\ind_{\scl}1^0}
&\Kt_{\comp}^1(T^*(Y\times\bbR)/(X\times\bbR))
\ar[d]^{\ind_{\scl}^1}\\
\Kt^0(X;\cA)\ar[r]&\Kt^1(X\times \bbR;\cA)
}
\label{mms2002.106}\end{equation}
where the inverses of the horizontal isomorphism are the Toeplitz index maps.
\end{proposition}

\begin{corollary}\label{mms2002.107} To prove the equality of the analytic
  and topological index maps it suffices to prove the equality of the odd
  semiclassical and odd topological index maps.
\end{corollary}

\begin{proof} Suppose we have proved the equality of odd semiclassical and
  odd topological index maps 
\begin{equation}
\xymatrix@1{
**[l]\Kt^1_{\comp}(T^*(Y/X))\ar@<1ex>[r]^-{\ind^1_{\scl}}
\ar@<-1ex>[r]^-=_-{\ind^1_{\text{t}}}&\Kt^1_{\comp}(X;\cA).
}
\label{mms2002.108}\end{equation}
Both the topological and the semiclassical index maps give commutative
diagrams as in Proposition~\ref{mms2002.104}, so it follows that the more
standard, even, versions of these maps are also equal.
\end{proof}

\begin{lemma}\label{mms2002.110} For an iterated fibration of manifolds 
\begin{equation}
\xymatrix{Z'\ar@{-}[r]&M'\ar[d]^{\psi}\\
Z\ar@{-}[r]&M\ar[d]^{\phi}\\
&B
}
\label{mms2002.111}\end{equation}
the semiclassical index gives a commutative diagram
\begin{equation}\xymatrix{
\Kco(T^*(M'/Y))\ar[dd]^{\ind_{\scl}^1}\ar[dr]^{\ind_{\scl}^1}\\
&\Kco(T^*(M/Y))\ar[dl]^{\ind_{\scl}^1}\\
\Kco(Y)
}
\label{mms2002.112}\end{equation}
where the map on the top right is the semiclassical index map for the
fibration of $M'$ over $M$ pulled back to $T^*(M/Y).$
\end{lemma}

\begin{proof} The commutativity of \eqref{mms2002.112} follows from the
use of a double semiclassical quantization, with different parameters in
the two fibres (see the extensive discussion in \cite{fmtga}).
\end{proof}

\begin{lemma}\label{mms2002.113} For any complex, or real-symplectic,
vector bundle $W$ over a manifold $X$ the semiclassical index implements
the Thom isomorphism 
\begin{equation}\xymatrix{
\Kco(W)\ar@<1ex>[r]^{\ind_{\scl}^1}_{=}\ar@<-1ex>[r]_\Thom&\Kco(X).
}
\label{mms2002.114}\end{equation}
\end{lemma}

\begin{proof} This again follows from the use of semiclassical quantization
  in the `isotropic' of (pseudodifferential) Weyl algebra of operators on a
  symplectic vector space. The resulting symbol map is shown, in
  \cite{fmtga}, to be an isomorphism using the argument of Atiyah. Since
  the Thom map constructed this way is homotopy invariant it applies to to
  case of a complex vector bundle where the `positive' sympectic structure
  on the underlying real bundle is fixed up to homotopy.
\end{proof}

\section{The index theorem}\label{sect:index thm}

The odd topological index is defined as the composite map arising from an
embedding so we wish to prove the commutativity of the diagram
\begin{equation}
\xymatrix{
\Kco(T^*(Y/X))\ar[r]\ar[rrd]_{\ind_{\scl}}&
\Kco(T^*(\Omega/X);\cA)\ar[r]^{\iota^*}\ar[rd]^{\ind_{\scl}}&
\Kco(T^*(\bbR^M)\times X);\cA)\ar@/^{-.5pc}/[d]^{=}_{\ind_{\scl}}
\ar@<1.5ex>[d]^{\Thom}\\
&&\Kco(X;\cA).
}
\label{mms2002.109}\end{equation}
Here $\Omega$ is a collar neighbourhood of $Y$ embedded in $\bbR^M,$ so is
isomorphic to the normal bundle to $Y.$ Thus, it suffices to prove
commutativity in three places. The last of these is equality of the two
maps on the right, that the semiclassical index map implements the Thom
isomorphism (or in this trivial case, Bott periodicity). The second is
`excision' which is immediate from the definition of the semiclassical
index. The first commutativity, for the triangle on the left corresponds to
multiplicativity of the semiclassical index which in this case reduces to
(a special case of) Lemma~\ref{mms2002.110}. 
 
This leads to the main theorem. Here we tacitly identify the tangent and cotangent
bundles via a Riemannian metric.

\begin{theorem}[The index theorem in $K $-theory]\label{K-indexthm} Let
  $\phi:Y\longrightarrow X$ be a fiber bundle of compact manifolds,
  together with the other data in \eqref{mms2002.115} --
  \eqref{mms2002.87}. Let $\cA$ be the smooth Azumaya bundle over $X$ as
  defined in \S\ref{SmoothAzbun} and $P\in\Psi^\bullet (Y/X,\cA\otimes\bbE)$ be a
  projective family of elliptic pseudodifferential operators acting on the
  projective Hilbert bundle $\bbP(\phi_*(\bbE \otimes K_\tau))$ over $X,$
  with symbol $p\in\Kc(T(Y/X)),$ then
\begin{equation}
\ind_a(P) = \ind_t(p) \in K^0(X, \cA).
\end{equation}
\end{theorem}

\section{The Chern character of the index}\label{sect:coh}

As discussed above, the index map in K-theory can be considered as acting
on the untwisted K-theory, with compact supports, of $T^*(Y/X),$ via the
identification with the (trivially) twisted K-theory coming from the
original choice of data \eqref{mms2002.115} -- \eqref{mms2002.87}. The
Chern character for the symbol class in the standard setting, 
\begin{equation}
\Kc^0(T^*(Y/X))\longrightarrow\operatorname{H}_{\text{c}}^{\even}(T^*(Y/X))
\label{mms2002.126}\end{equation}
can be represented explicitly in terms of symbol data and
connections in a relative version of the formul\ae\ in \S\ref{TwistedChern}
following Fedosov \cite{Fedosov}. A K-class is represented by bundles $(E_+,E_-)$
over $Y$ and an elliptic symbol $a$ identifying them over $S^*(Y/X).$ It is
convenient to use the relative interpretation of the cohomology from
\cite{Albin-Melrose}. Thus one can take the explicit representative 
\begin{equation}
\begin{gathered}
\Ch([(E_+,E_-,a)]=(\widetilde\Ch(a),\Ch(E_+)-\Ch(E_-)),\\
\widetilde\Ch(a)=-\frac1{2\pi i}
	\int_0^1 \tr\left(a^{-1}(\nabla a)e^{w(t)}\right)dt\Mwhere\\
w(t)=(1-t)F_++ta^{-1}F _-a + \frac 1{2\pi i} t(1-t)(a^{-1}\nabla a)^2\Mand\\
\Ch(E_{\pm})=\tr\exp(F_{\pm}/2\pi i),\quad F_{\pm}=\nabla_{\pm}^2.
\end{gathered}
\label{mms2002.127}\end{equation}
Here $\nabla_{\pm}$ are connections on $E_{\pm}$ over $Y$ and $\nabla$ is the
induced connection on $\hom(E_+,E_-)$ lifted to $S^*(Y/X).$ Note that the
underlying relative complex is the direct sum of the deRham complexes with
differential
\begin{equation}
\CI(S^*(Y/X);\Lambda ^*)\oplus\CI(Y;\Lambda ^*),\quad
D=\begin{pmatrix}
d&\pi^*
\\
0&-d
\end{pmatrix}.
\label{mms2002.128}\end{equation}

In our twisted case, as shown in \S\ref{TwistedChern} the line bundle $J$
over $Y^{[2]}$ decomposes
as $\tilde J\boxtimes \tilde J'$ when lifted to $\tilde Y^{[2]}$ which has
an additional fiber factor of $\tilde L.$  The discussion of the Chern
character therefore carries over directly to this relative setting.

\begin{proposition}\label{mms2002.131} For any element of $\Kc^0(T^*(Y/X))$
  represented by (untwisted) data $(E_+,E_-,a)$ the twisted Chern
  character of the image in $\Kc(T^*(Y/X);\rho^*\phi^*\cA)$ is represented by the
  pair of forms after lifting to $\tilde L$ and trivializing $J$ as in
  \eqref{mms2002.188}
\begin{equation}
\Ch_{ \rho^*\phi^*\cA}([(E_+,E_-,a)])=(\widetilde\Ch_{\cA}(a),\Ch_{\cA}(E_+)-\Ch_{\cA}(E_-))
\label{mms2002.132}\end{equation}
in the subcomplex of the relative deRham complex fixed by \eqref{mms2002.147},
and  $\rho\colon T^*(Y/X)\longrightarrow Y$ is the projection.
\end{proposition}

Of course the point of this discussion is that these forms do give the
analogue of the index formula in (twisted) cohomology.

\begin{theorem}\label{mms2002.134} For the twisted index map
  \eqref{mms2002.135} the twisted Chern character is given by the push-forward
  of the differential form in \eqref{mms2002.132} 
\begin{equation}
\begin{gathered}
\Ch_{\cA}\circ\ind:\Kc^0(T^*(Y/X);\rho^*\phi^*\cA)\simeq \Kc^0(T^*(Y/X))\longrightarrow
\operatorname{H}^{\even}(X, \delta),\\
\Ch_{\cA}\circ\ind(p)=
 (-1)^n \phi_*\rho_*\left\{\rho^*{\rm
Todd}(T^*(Y/X)\otimes \bbC)
\wedge
\Ch_ {\rho^*\phi^*\cA}(p)\right\},
\end{gathered}
\label{mms2002.136}\end{equation}
where ${\rm Todd}(T^*(Y/X)\otimes\bbC)$ denotes the 
Todd class of the complexified vertical cotangent bundle and $p=[(E_+,E_-,a)]$
as identified in Proposition \ref{mms2002.131}.
\end{theorem}

\begin{proof} 
 By the Index Theorem in K-theory, Theorem \ref{K-indexthm} of the previous section, it suffices 
 to compute the twisted Chern character of the topological index of the projective family
of elliptic pseudodifferential operators. We
begin with by recalling the basic properties of the twisted Chern character.
As before, we assume that the primitive line bundle $J$ defining the smooth Azumaya bundle 
$\cA$ is endowed with a fixed connection respecting the primitive property.
It gives a homomorphism,
\begin{equation}\label{Chern-char}
\Ch_\cA :\Kt^0(X, \mathcal{A})\longrightarrow \operatorname{H}^{even}(X, \delta),
\end{equation}
satisfying the following properties.

\begin{enumerate}
\item The Chern character is \emph{functorial} under smooth maps
in the sense that if $f\colon W \longrightarrow X$ is a smooth map between
compact manifolds, then the following diagram commutes:
\begin{equation} \label{natural2}
\begin{CD}
\Kt^0(X, \mathcal{A}) @>f^!>> \Kt^0(W, f^*\mathcal{A}) \\
                @VV{\Ch_\cA}V
@VV{\Ch_{f^*\cA}}V&
\\
\operatorname{H}^{even}(X, \delta) @>f^*>>   \operatorname{H}^{even}(W, f^*\delta).
\end{CD}
\end{equation}
Here the pullback primitive line bundle ${f^{[2]}}^*J$ defining the pullback 
smooth Azumaya bundle 
$f^*\cA$ is endowed with the pullback of the fixed connection respecting the primitive property.

\item The Chern character respects the structure of $\Kt^0(X,
\mathcal{A})$ as a module over $\Kt^0(X),$ in the sense that the following
diagram commutes:
\begin{equation} \label{module}
\begin{CD}
\Kt^0(X) \times\Kt^0(X, \mathcal{A}) @>>> \Kt^0(X, \mathcal{A}) \\
                @VV{\Ch \times \Ch_\cA}V
@VV{\Ch_{\cA}}V&
\\
\operatorname{H}^{even}(X, \bbQ) \times \operatorname{H}^{even}(X, \delta	) @>>>   \operatorname{H}^{even}(X, \delta	)
\end{CD}
\end{equation}
where the top horizontal arrow is the action of $\Kt^0(X)$ on $\Kt^0(X,\mathcal{A})$
given by tensor product and the bottom horizontal arrow is given by
the cup product.
\end{enumerate}


The theorem now follows rather routinely from the index theorem in $K$-theory,
Theorem~\ref{K-indexthm}. The key step to getting the formula
is the analog of the Riemann-Roch formula in the context of
twisted K-theory, which we now give details.

Let $\pi: E\longrightarrow X$ be a spin$\bbC$ vector bundle over $X$,
$i: X\longrightarrow E$ the zero section embedding, and
$F \in \Kt^0(X, \cA)$. Then using the properties of the twisted Chern character
as above, we compute,
$$
\begin{array}{lcl}
\Ch_{\pi^*\cA}(i_! F) &=& \Ch_{\pi^*\cA}(i_! 1 \otimes \pi^*F)\\[+7pt]
		 &=& \Ch(i_! 1) \wedge \Ch_{\pi^*\cA}(\pi^* F).
\end{array}
$$ 
The standard Riemann-Roch formula asserts that
$$
\Ch(i_! 1) = i_*{\rm Todd}(E)^{-1}.
$$
Therefore we deduce the following Riemann-Roch formula for twisted
$K$-theory,
\begin{equation}\label{RR}
\Ch_{\pi^*\cA}(i_! F) = i_*\left\{{\rm Todd}(E)^{-1} \wedge \Ch_\cA( F)
\right\}.
\end{equation}

We need to compute $\Ch_\cA (\ind_t p)$ where
$$
p= [E_+, E_-, a] \in \Kc^0(T(Y/X)) \cong 
\Kc^0(T(Y/X), \rho^*\phi^*\cA).
$$
We will henceforth identify $T(Y/X)\cong T^*(Y/X)$ via a Riemannian
metric. Recall from \S6 that the topological index, $\ind_t = j_!^{-1}
\circ (Di)_! $ where $i  : Y \hookrightarrow X\times \bbR^{2N}$ is an
embedding that commutes with the projections
$\phi : Y\longrightarrow X$ and $\pi_1 : X\times  \bbR^{2N} \longrightarrow X$, and
$j : X \hookrightarrow X\times  \bbR^{2N}$ is the zero section
embedding. Therefore
$$
\Ch_\cA (\ind_t p) = \Ch_\cA (j_!^{-1} \circ (Di)_! p)
$$
By the Riemann-Roch formula for twisted K-theory \eqref{RR},
$$
\Ch_{\pi_1^* \cA}(j_! F) =  j_*\Ch_\cA(F)
$$
since $\pi_1: X\times  \bbR^{2N} \longrightarrow X$ is a trivial bundle.
Since
${\pi_1}_* j_* 1 = (-1)^n$, it follows that
for $\xi \in \Kc^0(X \times  \bbR^{2N}, \pi_1^*\cA) $, one has
$$
\Ch_\cA(j_!^{-1} \xi) = (-1)^n{\pi_1}_*\Ch_{\pi_1^*\cA}( \xi)
$$
Therefore
\begin{equation}\label{a1}
\Ch_\cA (j_!^{-1} \circ (Di)_! p)
=  (-1)^n {\pi_1}_*\Ch_{\pi_1^*\cA}((Di)_! p)
\end{equation}
By the Riemann-Roch formula for twisted
$K$-theory \eqref{RR},
\begin{equation}\label{a2}
\Ch_{\pi_1^*\cA}((Di)_!  p) = (Di)_*\left\{\rho^*{\rm Todd}(\cN)^{-1}
\wedge\Ch_ {\rho^*\phi^*\cA}( p)\right\}
\end{equation}
where $\cN$ is the complexified normal bundle to the embedding
$Di : T(Y/X) \longrightarrow X \times T\bbR^{2N}$, that is,
$\cN =  X \times T\bbR^{2N}/Di(T(Y/X))\otimes \bbC$.
Therefore ${\rm Todd}(\cN)^{-1} = {\rm Todd}(T(Y/X)\otimes \bbC)$ and
          \eqref{a2} becomes
$$
\Ch_{\pi_1^*\cA}((Di)_! p) =(Di)_*\left\{\rho^*{\rm Todd}(T(Y/X)\otimes \bbC)
\wedge\Ch_ {\rho^*\phi^*\cA}( p)\right\}.
$$
Therefore \eqref{a1} becomes
\begin{equation}
\begin{array}{lcl}
\Ch_\cA (j_!^{-1} \circ (Di)_! p) &= &
(-1)^n {\pi_1}_*(Di)_*\left\{\rho^*{\rm Todd}(T(Y/X)\otimes \bbC)  \wedge
\Ch_ {\rho^*\phi^*\cA}( p)\right\}\\[+7pt]
& = & (-1)^n \phi_*\rho_*\left\{\rho^*{\rm Todd}(T(Y/X)\otimes \bbC)  \wedge
\Ch_ {\rho^*\phi^*\cA}( p)\right\}
\end{array}
\end{equation}
since $\phi_* \rho_*= {\pi_1}_* (Di)_*$. Therefore
\begin{equation}\label{cohtopindex}
\Ch_\cA (\ind_t p) = (-1)^n \phi_*\rho_*\left\{\rho^*{\rm
Todd}(T(Y/X)\otimes \bbC)
\wedge
\Ch_ {\rho^*\phi^*\cA}(p)\right\},
\end{equation}
proving Theorem~\ref{mms2002.134}.

\end{proof}

\appendix
\section{Differential characters}\label{diffchar}

We will refine Lemma \ref{mms2002.64} to an equality of differential
characters. For an account of differential characters, see \cite{CS, HS}.

We first relate a connection $\tilde\gamma$ on $\tilde L$ to the 1-form $\gamma$ on $Y$. 
Consider the commutative diagram 

\begin{equation}
\xymatrix{
\phi^*(\tilde L) \ar[r]_{\tilde \phi}\ar[d]_{pr_1} & {\tilde L} \ar[d]_{\pi}\\
Y\ar[r]_{\phi} & X
}
\end{equation}
where $\phi^*(\tilde L)=Y\times {\mathbb S}$ as observed earlier, and 
$pr_1 : Y\times {\mathbb S} \longrightarrow Y$ denotes projection to the 
first factor. Then the connection 1-form $\tilde \gamma$ on $\tilde L$ with curvature
equal to $\beta$ is 
related to the 1-form $\gamma$ on $Y$ by ${\tilde \phi}^*(\tilde\gamma)= 
\gamma + \theta$ where $\theta$ is the Cartan-Maurer 1-form on $\bbS$.
If $\iota : Y \to Y \times {\mathbb S}$ denotes the inclusion map into the first factor,
then $\iota^*{\tilde \phi}^*(\tilde\gamma) = \gamma$. 

Now the circle bundle $\tilde L$ has a section $\tilde \tau : X\setminus M_1 \to \tilde L$,
where $M_1$ is a codimension 2 submanifold of $X$. We define a section 
$\tau :  X\setminus M_1 \to Y$ such that $\tilde\phi\circ\iota\circ\tau = \tilde\tau$.
Then we have a well defined singular 1-form $\varphi_1:= \tilde\tau^*(\tilde \gamma)
= \tau^*(\gamma)$ on $X$ with the property that $d\varphi_1 = \beta$.
The differential character associated to $\varphi_1$ is (cf. \cite{Cheeger})
$$
S(\varphi_1)(z) = \varphi_1(z') + \beta(c)
$$
where $z, z' \in Z_1(X, \bbZ)$ and $c\in C_2(X, \bbZ)$ is such that 
$\partial c = z-z'$, where $z'\cap M_1=\emptyset$.

The smooth map $u:X\to \bbR/\bbZ$ gives rise to a singular function $\varphi_0$ on $X$ as follows. If
$t \in \bbR/\bbZ$ is a regular value for $u$, then $M_0:=u^{-1}(t)$ is a codimension $1$ submanifold  of $X$, and
the Cartan-Maurer $1$-form $\theta$ on $\bbR/\bbZ$ is exact on $\bbR/\bbZ\setminus\{t\}$, say 
$dg$, where $g$ is a smooth function
on $\bbR/\bbZ\setminus\{t\}$. Then the pullback function $\varphi_0=u^*(g)$ is a smooth function on 
$X\setminus M_0$, ie it is a singular function
on $X$ such that $d\varphi_0= u^*(\theta)=\alpha$ is the associated global smooth $1$-form on $X$ with integer periods.

With $\mu$ as in Lemma~\ref{mms2002.64}, $\varphi_2:=\tau^*(\mu)= d\varphi_0\wedge\varphi_1$ 
is a singular 2-form on $X$, whose associated differential character is
$$
S(\varphi_2)(z) = d\varphi_0\wedge\varphi_1(z') + \alpha\wedge\beta(c)
$$
where $z, z' \in Z_2(X, \bbZ)$ and $c\in C_3(X, \bbZ)$ is such that 
$\partial c = z-z'$, where $z'\cap M_1=\emptyset$.
By the argument given above, it is also the differential character associated to 
the Azumaya bundle $\cA$ with connection.

\begin{lemma}
In the notation above, $S(\varphi_2) = S(\varphi_0)\star S(\varphi_1)$, where $\star$ denotes the
Cheeger-Simons product of differential characters. 
\end{lemma}

\begin{proof}  First note that by \cite{CS}, the field strength of $S(\varphi_0)\star
S(\varphi_1)$ is $\bar\alpha \wedge \bar\beta$, which by Lemma~\ref{mms2002.64}
  is equal to $\bar\delta$ which is the field strength of $S(\varphi_2)$.  
  That is, $$S(\varphi_0)\star S(\varphi_1) (\partial c) = S(\varphi_2) (\partial c)$$
for every degree 3 integral cochain $c$.   
  
  By
  \cite{CS}, we see that the characteristic class of $S(\varphi_0) \star S(\varphi_1)$ is
  equal to the cup product $\alpha\cup\beta$, and by Appendix
  \ref{appendix:cech}, also equal to $\delta$, which is the characteristic
  class of $S(\varphi_2).$ Note that the image in real cohomology of $\alpha$ and
  $\beta$ is equal to $[\bar\alpha]$ and $[\bar\beta]$ respectively.
 
According to \cite{Cheeger}, if $z \in Z_2(X, \bbZ)$ is transverse to $M_1$, then
$$
S(\varphi_0)\star S(\varphi_1)(z) = - d\varphi_0\wedge\varphi_1(z) + \sum_{p\in z\cap M_1}\varphi_0(p)
$$
 In particular, if $z\cap M_1 = \emptyset$, then 
 $$
S(\varphi_0)\star S(\varphi_1)(z) = S(\varphi_2)(z),
$$
proving the lemma.
   \end{proof}

\section{\v Cech class of the Azumaya bundle}\label{appendix:cech}

Suppose that there is a line bundle $K$ over $Y$ such that $J \cong K
\boxtimes K'.$ That is, in this case, $\CI(X, \cA)\cong \Psi^{-\infty}(Y/X,
K)$ is the algebra of smoothing operators along the fibres of
$\phi:Y\longrightarrow X$ acting on sections of $K,$ and therefore has
trivial Dixmier-Douady invariant. Here $\cA_x= \Psi^{-\infty}(\phi^{-1}(x),
K\big|_{\phi^{-1}(x)})$ for all $x\in X.$

Conversely, suppose that the Dixmier-Douady invariant of $\cA$ is trivial,
where $\CI(X, \cA) = \CI(Y^{[2]}, J).$ Then there is a line bundle $K$ over
$Y$ such that $J \cong K \boxtimes K'.$ To see this, we use the connection
$\nabla^J $ preserving the primitive property of $J,$ and Lemma
\ref{mms2002.64} to see that $d\mu=\phi^*dB$, for some global 2-form $B\in
\Omega^2(X).$ Then $d(\mu - \phi^*(B))=0$, and $\pi_1^*(\mu - \phi^*(B)) -
\pi_2^*(\mu - \phi^*(B)) = F_{\nabla^J}.$ So $\mu - \phi^*(B)$ is a closed
2-form on $Y$ and can be chosen to have integral periods, since
$F_{\nabla^J}$ has integral periods (this is clear from the \v Cech
description below).  Therefore there is a line bundle $K$ on $Y$ with
connection, whose curvature is equal to $\mu - \phi^*(B)$ such that $J
\cong K\boxtimes K'.$

Suppose that $J_1$ and $J_2$ are two primitive line bundles over $Y^{[2]}$
and let $\cA_1$ and $\cA_2$ be the corresponding Azumaya bundles, that is,
$\CI(X, \cA_j) = \CI(Y^{[2]}, J_j),$ $j=1,2$. Then we conclude by the
argument above that $\cA_1 \cong \cA_2$ if and only if there is a line
bundle $K$ over $Y$ such that $J_1 \cong J_2 \otimes (K \boxtimes K').$

The main result that we want to show here is the following.

\begin{lemma} Suppose $\phi:Y\longrightarrow X,$ $L\longrightarrow X$ and
  $u:X\longrightarrow \UU(1)$ are as in the introduction.  Then the
  Dixmier-Douady class of the Azumaya bundle $\cA$ constructed from this
  data as in \S\ref{SmoothAzbun}, is equal to $\alpha \cup\beta$, where
  $\alpha \in \operatorname{H}^1(X, \bbZ)$ is the cohomology defined by $u$ and $\beta \in
  \operatorname{H}^2(X, \bbZ)$ is the Chern class of $L.$
\end{lemma}

\begin{proof} As noted above, the Dixmier-Douady invariant of $\cA$ is the
  degree 3 cohomology class on $X$ associated to the primitive line bundle
  $J$ over $Y^{[2]}.$

As argued in \S\ref{decompGeom}, the line bundle $L$ gives rise to a
character $s:Y^{[2]} \to \UU(1)$. Suppose that $\tau_i:U_i \to Y$ are local
sections of $Y$. Then it is clear from \S\ref{decompGeom} that $c_{ij} :=
s(\tau_i, \tau_j)$ defines a $ \UU(1)$-valued Cech 1-cocycle representing
the first Chern class of $L$.

Using the same local sections of $Y$, we see that $J_{ij}:=(\tau_i  
\times \tau_j)^*J = L_j^{-n_{ij}}$,
where $n_{jk}: U_j\cap U_k\to \bbZ$ denotes the transition functions  
of $\hat X$.
If $s_j$ is a local nowhere zero section of $L_j$, then $ 
\sigma_{ij} :=s_j^{-n_{ij}}$ is a local
nowhere zero section of $J_{ij}$. We compute,
\begin{align}
\sigma_{ij}\sigma_{jk} &= s_j^{-n_{ij}} s_k^{-n_{jk}}\\
& = c_{jk}^{-n_{ij}}  s_k^{-n_{ij}}  s_k^{-n_{jk}}\\
& = c_{kj}^{-n_{ji}} s_k^{-n_{ik}} = c_{kj}^{-n_{ji}} \sigma_{ik}.
\end{align}
Therefore the $\UU(1)$-valued Cech 2-cocycle associated to $J$ is  
$d_{ijk} :=  c_{kj}^{-n_{ji}}$
But it is well known  (cf\@.~ equation (1-18), page 29, \cite{Bry})
that the right hand side represents the cup product
of the Cech cocycles $[c]$ and $[-n]$, that is, $[d]= [c] \cup [-n]
=- \beta \cup\alpha = \alpha\cup\beta\in \operatorname{H}^3(X, \bbZ)$,
proving the lemma.
\end{proof}

\section{The universal case}\label{appendix:universal}

Let $\phi: Y \rightarrow X$ be a fibre bundle of compact manifolds,
$L\to X$ a line bundle over $X$ with the property that the pullback
$\phi^*(\beta)=0$ in $\operatorname{H}^2(Y, \bbZ)$, where $\beta \in \operatorname{H}^2(X, \bbZ)$
is the first Chern class of $L$.

\begin{lemma}\label{append:univ}
In the notation above,  $\phi^*(\beta)=0$ in $\operatorname{H}^2(Y, \bbZ)$ if and  
only if
there is a $\tilde \beta \in \operatorname{H}^2(B{\rm Diff}(Z), \bbZ)$ such that
$\beta= f^*(\tilde\beta)$ in $\operatorname{H}^2(X, \bbZ)$,
where $f\colon~X~\to~B{\rm Diff}(Z)$ is the classifying map for $ 
\phi:Y \to X$, and
$Z$ is the typical fiber of $\phi:Y \to X.$
\end{lemma}

This follows in a straightforward way from standard algebraic  
topology. The direction that
we will mainly use is trivial to prove, viz. if
there is a class $\tilde \beta \in \operatorname{H}^2(B{\rm Diff}(Z), \bbZ)$ such that
$\beta= f^*(\tilde\beta)$ in $\operatorname{H}^2(X, \bbZ)$, then $\phi^*(\beta)=0$ in  
$\operatorname{H}^2(Y, \bbZ)$.

Therefore we see that given any fibre bundle of compact manifolds
$\phi: Y \rightarrow X$ with typical fiber $Z$, and $\beta \in   
f^*(\operatorname{H}^2(B{\rm Diff}(Z), \bbZ))
\subset \operatorname{H}^2(X, \bbZ)$ (that is, if $\beta$ is a characteristic class  
of the fiber bundle
$\phi:Y \to X$), then $\phi^*(\beta)=0$ in $\operatorname{H}^2(Y, \bbZ)$,
satisfying the hypotheses of our main index theorem.

But what are line bundles on $B{\rm Diff}(Z)$? Since roughly speaking,
$B{\rm Diff}(Z) = {\rm Metrics}(Z)/{\rm Diff}(Z)$, where ${\rm Metrics} 
(Z)$ denotes
the contractible space of
all Riemanian metrics on $Z$, the theory of anomalies in gravity  
constructs
line bundles on $B{\rm Diff}(Z)$ via determinant
line bundles of index bundles of families of twisted Dirac operators  
obtained
by varying the Riemannian metric on $Z$, cf\@.~\cite{ASZ}.

In particular,  let $\phi: Y \rightarrow X$ be a fibre bundle of compact  
manifolds,
with typical fiber a compact Riemann surface $\Sigma_g$
of genus $g\ge 2$. Then $T(Y/X)$ is an oriented rank 2 bundle over $Y.$
Define $\beta = \phi_*(e \cup e) \in \operatorname{H}^2(X, \bbZ)$, where
$e := e(T(Y/X)) \in \operatorname{H}^2(Y, \bbZ)$ is the Euler class of $T(Y/X)$. By  
naturality of this
construction, $\beta = f^*(e_1)$, where $e_1 \in \operatorname{H}^2(B{\rm Diff} 
(\Sigma_g), \bbZ)$
and $f\colon X \to B{\rm Diff}(\Sigma_g)$ is the classifying map for $ 
\phi:Y \to X$.
$e_1$ is known as the universal first Mumford-Morita-Miller class,
and $\beta$ is the first Mumford-Morita-Miller class of $\phi: Y  
\rightarrow X$, cf\@.~Chapter 4 in \cite{Morita}.
Therefore by Lemma \ref{append:univ}, we have the following.

\begin{lemma}
In the notation above, let $\phi: Y \rightarrow X$ be a fibre bundle of  
compact manifolds,
with typical fiber a compact Riemann surface $\Sigma_g$
of genus $g\ge 2$, and let $\beta \in \operatorname{H}^2(X, \bbZ)$ be a multiple of the
first Mumford-Morita-Miller class of $\phi: Y \rightarrow X$.
Then $\phi^*(\beta) = 0$ in $\operatorname{H}^2(Y, \bbZ)$.
\end{lemma}

Such choices satisfy the hypotheses of our main index theorem. In  
fact, if $\phi:Y\longrightarrow X$
be as above, and in addition let $X$ be a closed Riemann surface.
Then Proposition 4.11 in \cite{Morita} asserts that $\langle e_1, [X] 
\rangle = {\rm Sign}(Y)$, where
$ {\rm Sign}(Y)$ is the signature of the 4-dimensional manifold $Y$,  
which is originally a result of
Atiyah. As a consequence, Morita is able to produce infinitely many  
surface bundles $Y$ over
$X$ that have non-trivial  first Mumford-Morita-Miller class.

On the other hand, given any $\beta \in \operatorname{H}^2(X, \bbZ)$, we know that  
there is
a fibre bundle of compact manifolds
$\phi: Y \rightarrow X$ such that $\phi^*(\beta)=0$ in $\operatorname{H}^2(Y, \bbZ)$.
In fact we can choose $Y$ to be the total space of a principal ${\rm U} 
(n)$ bundle
over $X$ with first Chern class $\beta$.  Here we can also replace $ 
\UU(n)$
by any compact Lie group $G$ such that $\operatorname{H}^1(G, \bbZ)$ is nontrivial  
and torsion-free,
such as the torus $\bbT^n$.

\begin{lemma}
Let $\phi: Y \rightarrow X$ be a fibre bundle of compact manifolds with  
typical fiber
$Z$ and $\beta \in \operatorname{H}^2(X, \bbZ)$.
Let $\pi: P \to X$ be a principal ${\rm U}(n)$-bundle whose first  
Chern class is $\beta$. Then the fibred
product  $\phi \times \pi : Y\times_X P \to X$ is a fiber bundle with  
typical fiber $Z \times {\rm U}(n)$,
and has the property that $(\phi \times \pi )^*(\beta) = 0$ in $\operatorname{H}^2(Y 
\times_X P, \bbZ).$
\end{lemma}

This follows from the obvious commutativity of the
  following diagram,
  \begin{equation}
\begin{CD}
Y\times_X P  @>{pr_1}>> Y \\
       @V{pr_2}VV          @VV{\phi}V     \\
P   @>\pi>>  X.
\end{CD}\end{equation}
Hence this data also satisfy the hypotheses of our main index theorem.

The construction of the universal fibre bundle of Riemann surfaces which we
will describe next, is well known, cf. \cite{ASZ, AS84, F86}.  Let $\Sigma$
be a compact Riemann surface of genus $g$ greater than $1,$ $\fM_{(-1)}$ the
space of all hyperbolic metrics on $\Sigma$ of curvature equal to $-1,$ and
${\rm Diff}_+(\Sigma)$ the group of all orientation preserving
diffeomorphisms of $\Sigma$. Then the quotient
$$
\fM_{(-1)}/{\rm Diff}_+(\Sigma) =  \cM_g
$$ 
is a noncompact orbifold, namely the moduli space of Riemann surfaces of 
genus equal to $g$.
The fact that $\cM_g$ has singularities can be dealt with in several ways, 
for instance by going to a finite smooth cover, and the noncompactness of $\cM_g$
can be dealt with for instance by considering compact submanifolds. We will 
however not deal with these delicate issues in the discussion below.
The group ${\rm Diff}_+(\Sigma)$
also acts on $\Sigma \times \fM_{(-1)}$ via $g(z, h) = (g(z), g^*h)$
and the resulting smooth fibre bundle, 
\begin{equation}
\pi: Y = (\Sigma \times \fM_{(-1)})/{\rm Diff}_+(\Sigma) \longrightarrow 
\fM_{(-1)}/{\rm Diff}_+(\Sigma) =\cM_g
\label{UniBun}
\end{equation}
is the {\em universal bundle} of genus $g$ Riemann surfaces. The classifying map for
\eqref{UniBun} is the identity map on $\cM_g$ so $\pi$ is maximally nontrivial in a sense made
precise below.  

As defined above, let
$$
e_1=e_1(Y/\cM_g) = \pi_*(e\cup e) \in H^2(\cM_g; \bbZ)
$$
be the first Mumford-Morita-Miller class of $\pi\colon 
 Y \to \cM_g.$

A theorem of Harer \cite{Harer, Morita} asserts that:
\begin{align*}
H^2(\cM_g; \bbQ) & = \bbQ(e_1);\\ 
H^1(\cM_g; \bbQ) & =\{0\}.
\end{align*}
Our next goal is to define 
a line bundle $\cL$ over $\cM_g$ such that $c_1(\cL) =k e_1$ for some 
$k\in \ZZ$. This line 
bundle then automatically has the property that $\pi^*(\cL)$ is trivializable since $e_1$ 
is a characteristic class of the fibre bundle $\pi:Y\longrightarrow\cM_g$.
This is exactly
the data that is needed to define a projective family of Dirac operators. The line bundle
$\cL$ turns out to be a power of the determinant line bundle of the virtual vector bundle $\Lambda$ 
known as the Hodge bundle, which is defined 
using the Gysin map in K-theory. 
$$\Lambda= \pi_!(T(Y/ \cM_g)) \in K^0( \cM_g).$$ Then  
$\det(\Lambda)$ is actually a line bundle over $\cM_g$. 
Next we need the following special Grothendieck-Riemann-Roch (GRR) calculation.

\begin{lemma}
In the notation above, one has the following identity of first Chern classes,
$$
 c_1( \pi_!(T(Y/\cM_g)) 
 =\frac{13}{12} \pi_*(c_1(T(Y/\cM_g))^2).
 $$
 \end{lemma}
 
 \begin{proof}
 By the usual GRR calculation \cite{AH}, we have
 $$
{\rm  ch}( \pi_!(T(Y/\cM_g)) 
 =\pi_*\left({\rm Todd}(T(Y/\cM_g)\cup {\rm ch}(T(Y/\cM_g))\right).
 $$
 Now 
 $$
 {\rm Todd}(x) = 1 + \frac{x}{2} + \frac{x^2}{12} +\ldots
 $$
 and 
 $$
{\rm  ch}(x) = 1 + x +\frac{x^2}{2} +\ldots
 $$
 where $x= c_1(T(Y/\cM_g))$. Therefore the degree 4 component is
 $$ 
\left[ {\rm Todd}(x) {\rm ch}(x)\right]_{(4)} = \frac{13}{12} x^2.
 $$
 That is, the  degree 2 component of the GRR formula in our case is 
 $$
 c_1( \pi_!(T(Y/\cM_g)) = \frac{13}{12} \pi_*(x^2).
 $$
 \end{proof}
 
Observing that $c_1(T(Y/\cM_g) = e$ and 
$$
c_1( \pi_!(T(Y/\cM_g))  = c_1(\Lambda) = c_1(\det(\Lambda)),
$$
the lemma above shows that $c_1(\det(\Lambda)) = \frac{13}{12} e_1.$
Setting $\cL = \det(\Lambda)^{\otimes 12},$ we obtain

\begin{corollary}
In the notation above, $\cL$ is a line bundle over $\cM_g$
and one has the following identity:
$$
 c_1(\cL) 
 = 13 e_1.
 $$
 \end{corollary}

We next construct a canonical projective family of Dirac operators
on the Riemann surface $\Sigma$.
We enlarge the parametrizing space by taking the product with the circle $\TT$.
Applying the main construction in the paper, we get
a primitive line bundle $J\longrightarrow Y^{[2]}$, where we denote 
the pullback of $Y$ over $\TT\times \cM_g$ 
by the same symbol. By the construction at the end of \S5, we obtain
a projective family of Dirac operators $\eth_J$ on the Riemann surface
$\Sigma$, parametrized by 
$\TT\times \cM_g$, having analytic index,
$$
\Index_a(\eth_J) \in K^0(\TT\times\cM_g; a \cup e_1),
$$
where $a \in H^1(\TT; \bbZ)$ is the generator.

\end{document}